\documentclass[10pt, a4paper]{amsart}

\usepackage[latin1]{inputenc}
\usepackage{amsfonts}
\usepackage{amsmath}
\usepackage{amssymb}
\usepackage{amsthm}
\usepackage[T1]{fontenc}
\usepackage{enumerate}
\usepackage{verbatim}
\usepackage{lmodern}
\usepackage{tensor}

\RequirePackage[dvipsnames,usenames]{color}  
\definecolor{VeryDarkGreen}{rgb}{0,0.18,0.08}
\definecolor{VeryDarkBrown}{rgb}{0.12,0.08,0.04}
\usepackage[pdftex]{hyperref}
\hypersetup{
bookmarks,
bookmarksdepth=2,
bookmarksopen,
bookmarksnumbered,
pdfstartview=FitH,
colorlinks,backref,hyperindex,
linkcolor=VeryDarkBrown,
citecolor=VeryDarkGreen,
urlcolor=VeryDarkGreen
}
\usepackage[matrix, arrow]{xy}
\usepackage{listings}
\title{Test module filtrations for unit $F$-modules}
\author{Axel St\"abler}

\DeclareMathOperator{\Supp}{Supp}

\DeclareMathOperator{\Hom}{Hom}

\DeclareMathOperator{\id}{id}
\DeclareMathOperator{\coker}{coker}
\DeclareMathOperator{\colim}{colim}

\DeclareMathOperator{\Tr}{Tr}

\DeclareMathOperator{\Ann}{Ann}

\DeclareMathOperator{\Spec}{Spec}

\DeclareMathOperator{\im}{im}

\DeclareMathOperator{\nil}{nil}

\newcommand{\eps}{\varepsilon}
\input xy
\xyoption{all}
\SelectTips{cm}{10}
\begin{document}
\swapnumbers
\theoremstyle{plain}
\newtheorem{Le}{Lemma}[section]
\newtheorem{Ko}[Le]{Corollary}
\newtheorem{Theo}[Le]{Theorem}
\newtheorem*{TheoB}{Theorem}
\newtheorem{Prop}[Le]{Proposition}
\newtheorem*{PropB}{Proposition}
\newtheorem{Con}[Le]{Conjecture}
\theoremstyle{definition}
\newtheorem{Def}[Le]{Definition}
\newtheorem*{DefB}{Definition}
\newtheorem{Bem}[Le]{Remark}
\newtheorem{Bsp}[Le]{Example}
\newtheorem{Be}[Le]{Observation}
\newtheorem{Sit}[Le]{Situation}
\newtheorem{Que}[Le]{Question}
\newtheorem{Dis}[Le]{Discussion}
\newtheorem{Prob}[Le]{Problem}
\newtheorem*{Konv}{Conventions}

\def\cocoa{{\hbox{\rm C\kern-.13em o\kern-.07em C\kern-.13em o\kern-.15em
A}}}

\address{Axel St\"abler\\
Johannes Gutenberg-Universit\"at Mainz\\ Fachbereich 08\\
Staudingerweg 9\\
55099 Mainz\\
Germany}
\email{staebler@uni-mainz.de}

\date{\today}

\subjclass[2010]{Primary 13A35; Secondary 14F10, 14B05}

\begin{abstract}
We extend the notion of test module filtration introduced by Blickle for Cartier modules. We then show that this naturally defines a filtration on unit $F$-modules and prove that this filtration coincides with the notion of $V$-filtration introduced by Stadnik in the cases where he proved existence of his filtration. We also show that these filtrations do not coincide in general.

Moreover, we show that for a smooth morphism $f: X \to Y$ test modules are preserved under $f^!$. We also give examples to show that this is not the case if $f$ is finite flat and tamely ramified along a smooth divisor.
\end{abstract}

\maketitle
\section*{Introduction}
If $R = \mathbb{C}[x_1, \ldots, x_n]$ and $f \in R$ is a hypersurface then the multiplier ideal filtration $\mathcal{J}(R, f^t)_{t \in \mathbb{Q}_{\geq 0}}$ is a descending right-continuous filtration of ideals in $R$ which captures subtle information about the singularities of $f$. This filtration is obtained by considering a log-resolution of $f$. Due to work of Budur and Saito (\cite{budursaitomultiplieridealvfiltration}, \cite{budurvfiltrationdmodules}) it is known that it captures parts of the information provided by the $V$-filtration. More precisely, if $\gamma: \Spec R \to \Spec R[t]$ is the graph embedding $t \mapsto f$ and $\gamma_+ R$ is the $\mathcal{D}_R$-module pushforward then the $V$-filtration of $\gamma_+ R$ intersected with $R$ yields the multiplier ideal filtration.
The associated graded of the $V$-filtration (constructed by Malgrange (\cite{malgrangevanishingdmodule}) and more generally by Kashiwara (\cite{kashiwaravfiltration})) corresponds to the nearby cycles functor in the category of of $\mathcal{D}_R$-modules.

Let us now assume for the rest of the introduction that $R$ is smooth over a perfect field. Then we have the theory of so-called test modules $\tau(M, f^t)_{t \in \mathbb{Q}_{\geq 0}}$ at our disposal, where $M$ is a coherent $R$-module endowed with an $R$-linear map $\kappa: F_\ast M \to M$.

For example, if $R = \mathbb{F}_p[x_1, \ldots, x_n]$ and  $M = R$, then $F_\ast R$ is free with basis $x_1^{i_1} \cdots x_n^{i_n}$ and $0 \leq i_j \leq p-1$. Hence one can define a map $\kappa: F_\ast R \to R$ sending $x_1^{p-1} \cdots x_n^{p-1}$ to $1$ and all other basis elements to zero (if one identifies $\omega_R$ with $R$ via $dx_1 \wedge \cdots \wedge dx_n \mapsto 1$ then this is the Cartier operator).  In this case the corresponding test ideal filtration $\tau(R, f^t)_{t \in \mathbb{Q}_{t \geq 0}}$ is then the characteristic $p$ analogue of the multiplier ideal filtration. 

The pairs $(M, \kappa: F_\ast M \to M)$ form a category which we call Cartier modules. It is an abelian category and it admits a natural functor to the category of $\mathcal{D}_R$-modules whose essential image are so-called unit $R[F]$-modules. In fact, suitably localizing the category of Cartier modules one can turn this into an equivalence with unit $R[F]$-modules which are special well-behaved $\mathcal{D}_R$-modules. Moreover, by work of Emerton and Kisin \cite{emertonkisinrhunitfcrys} one has an equivalence of unit $R[F]$-modules with perverse constructible $\mathbb{F}_p$-sheaves on the \'etale site associated to $R$.
Since nearby cycles themselves are not well-behaved in this category (it is not a functor on constructible sheaves among other things) it is now a very natural question whether one may construct a $V$-filtration in the category of unit $R[F]$-modules or Cartier modules that captures the desirable properties of the $V$-filtration in characteristic zero. In \cite{stadnikvfiltrationfcrystal} Stadnik introduced a notion of $V$-filtration in the category of unit $R[F]$-modules, proved existence in a special case and showed that in this case the zeroth graded piece corresponds to the unipotent part of the (underived) tame nearby cycles. In \cite{staeblertestmodulnvilftrierung} the author showed that in the category of Cartier modules one may characterize the test module filtration by certain axioms that are akin to a $V$-filtration and showed that if $(M, \kappa)$ corresponds to a locally constant sheaf and $x$ defines a smooth hypersurface then $Gr_\tau M$ (the associated graded of the test module filtration in the range $[0,1]$) is naturally isomorphic to $i^! M$ as a crystal, where $i: \Spec R/(x) \to \Spec R$ is the inclusion. Moreover, the author also showed that if $\varphi: \Spec S \to \Spec R$ is \'etale then $\varphi^! Gr_\tau M \cong Gr_\tau \varphi^!M$. Both results hold for nearby cycles in characteristic zero.

In the present paper we achieve two things. We extend the definition of test module filtration to Cartier modules that are obtained as the pushforward of a coherent Cartier module along an open immersion. In fact, we show that the definition of test module may be carried over to unit $R[F]$-modules. This enables us to compare the test module filtration with Stadnik's notion of $V$-filtration. We will show that these two filtrations coincide for the cases where Stadnik proved existence but not in general. Moreover, we show that for $\varphi: \Spec S \to \Spec R$ a smooth morphism one has $\varphi^! Gr_\tau M \cong Gr_\tau \varphi^!M$ (as one would expect from well-behaved nearby cycles).

We now describe the contents of this paper in more detail. While working with Cartier modules we will be able to relax our assumptions to $R$ being $F$-finite (that is the Frobenius $R \to F_\ast R$ is a finite morphism). Whenever the equivalence with unit $R[F]$-modules is involved we have to assume, in addition, that $R$ is smooth over an $F$-finite field. In the first section we review the theory of Cartier modules and unit $R[F]$-modules. Given a Cartier module $M$ we obtain a unit $R[F]$-module by considering $\colim {F^e}^! M \otimes \omega_R^{-1} = \mathcal{M}$. There is natural map $M \otimes \omega_R^{-1} \to \mathcal{M}$. We show that if two Cartier modules $M, N$ define the same unit $R[F]$-module then the images $\tau(M, f^t)$ and $\tau(N, f^t)$ in $\mathcal{M}$ coincide (Theorem \ref{UnitTestModuleFiltration}). This will be accomplished in Section \ref{SectionOne} after several preparations. Next, we extend the notion of test module filtration in Section \ref{SectionTestModuleFiltrationExtension} in order to be able to compare it with Stadnik's notion of $V$-filtration. Namely, if $j: D(f) \to \Spec R$, where $f \in R$, is an open immersion and $(M, \kappa)$ a coherent Cartier module then $j_\ast M$ is of course not coherent in general. However, there are coherent $R$-Cartier submodules $N \subseteq j_\ast M$ for which the inclusion is a local nil-isomorphism. We then show that $\tau(N, f^t)$ is independent of the choice of $N$ and use this to attach a test module filtration to $j_\ast M$.

Section \ref{ShriekFregularity} deals with functorial behavior of $F$-regularity and test modules. We show that for a smooth morphism $\varphi: \Spec S \to \Spec R$ one has $\varphi^! \tau(M, f^t) = \tau(\varphi^!M, f^t)$. This is a local problem. Since the \'etale case was handled in \cite{staeblertestmodulnvilftrierung} this boils down to understanding how $\varphi^!$ transforms the test module for the map $\varphi: \mathbb{A}^n_R \to \Spec R$. Once one has analyzed how the structural map $F_\ast \varphi^! M \to \varphi^! M$ looks like this is an explicit computation. We also give examples that this no longer holds if $\varphi$ is finite and tamely ramified along a smooth divisor. 

Finally in Section \ref{VFiltComparison} we achieve the comparison with Stadnik's $V$-filtration. We show that these two filtrations do not coincide in general (Example \ref{StadnikVFiltrationNotEqualTestModuleFiltration}) and show that they coincide in the cases where Stadnik proved existence. We also show that the zeroth graded piece of Stadnik's filtration which carries a unit $F$-crystal structure corresponds to the associated crystal of the first graded piece of the test module filtration. Both statements are part of Theorem \ref{VFiltTestmoduleMainComparison}.

As usual, whenever the upper shriek functor is involved there are a lot of compatibilities to be checked. Eventually all the compatibilities involving upper shriek and Cartier structures will hopefully be contained in an updated version of \cite{blickleboecklecartiercrystals}. In order to keep this paper reasonably self-contained we prove all compatibilities involving upper shriek and Cartier structures whenever there is no direct reference. In doing this we do not strive for maximal generality but rather restrict to the cases that we really need.

\subsection*{Conventions}
We assume that our rings are of positive prime characteristic $p > 0$. If $R$ is a ring then $R^\circ$ denotes the elements of $R$ that are not contained in any minimal prime. 

Throughout $F$ denotes the absolute Frobenius morphism. A noetherian ring $R$ is called $F$-finite if the Frobenius $F: \Spec R \to \Spec R$ is a finite morphism. Recall that by a result of Kunz \cite{kunznoetherian} an $F$-finite ring is excellent.

For a rational number $t$ we denote its round down by $\lfloor t \rfloor = \max \{ n \in \mathbb{Z} \, \vert \, n \leq t\}$, by $\{t\} = t - \lfloor t \rfloor$ its fractional part and its round up by $\lceil t \rceil = \min \{ n \in \mathbb{Z} \, \vert \, n \geq t\}$.

\subsection*{Acknowledgements}
The author was supported by SFB/Transregio 45 Bonn-Essen-Mainz financed by Deutsche Forschungsgemeinschaft. I thank Mircea Musta\c{t}\u{a} and Kevin Tucker for helpful discussions. Moreover, I thank Manuel Blickle for many useful discussions which among other things started this project. Finally, I thank the referee for a very thorough reading of this paper and many useful comments.

\section{Cartier modules and crystals}
Let $R$ be an $F$-finite ring. An \emph{$R$-Cartier algebra} $\mathcal{C}$ is an $\mathbb{N}$-graded not necessarily commutative $R$-algebra such that $R$ surjects onto $\mathcal{C}_0$ and such that we have the relation $r \varphi = \varphi r^{p^e}$ for all $\varphi \in \mathcal{C}_e$. A \emph{Cartier module} $M$ is a left-$\mathcal{C}$-module. We will also denote this by $(M, \mathcal{C})$. Note that if $\kappa_e \in \mathcal{C}_e$ is homogeneous of degree $e$ then multiplication by $\kappa_e$ is an $R$-linear map $F_\ast^e M \to M$. A Cartier module is called \emph{coherent} if its underlying $R$-module is finitely generated. The \emph{support} $\Supp M$ of a Cartier module $M$ is the support of the underlying $R$-module.

We will usually assume that Cartier modules are coherent with the notable exception of pushforwards along open immersions. In particular, we will omit the coherence assumption from the notation.

We note that the localization of a Cartier module is again a Cartier module in a natural way via the formula $\kappa_e(\frac{m}{s}) = \frac{\kappa_e(ms^{p^e -1})}{s}$.

A Cartier module $(M, \mathcal{C})$ is called \emph{nilpotent} if $(\mathcal{C}_+)^h M = 0$ for some $h \geq 0$. Given a coherent Cartier module $(M, \mathcal{C})$ there is a unique nilpotent submodule $M_{nil}$ such that $\overline{M} = M/M_{nil}$ does not admit nilpotent submodules (cf. \cite[Lemma 2.12]{blicklep-etestideale} and note the typing error: ``quotients`` should read ``submodules``).

Given a coherent Cartier module $(M, \mathcal{C})$ we define $\underline{M}_{\mathcal{C}} = (\mathcal{C}_+)^h M$ for $h \gg 0$ (\cite[Proposition 2.13]{blicklep-etestideale} shows that this is well-defined). We call $M$ \emph{$F$-pure} if $M = \underline{M}_{\mathcal{C}}$. We will often omit $\mathcal{C}$ from the notation and simply write $\underline{M}$. The Cartier module $\underline{M}$ is the unique Cartier submodule such that the quotient $M/\underline{M}$ is nilpotent and such that $\mathcal{C}_+ \underline{M} = \underline{M}$. The operation $\underline{\phantom{M}}$ commutes with localization and if $(M, \mathcal{C})$ is $F$-pure then its annihilator is a radical ideal (see \cite[Lemma 2.19]{blicklep-etestideale}).

A coherent Cartier module $(M, \mathcal{C})$ is called \emph{$F$-regular} if $(M, \mathcal{C})$ is $F$-pure and if it admits no non-zero proper Cartier submodules that generically agree with $M$ (that is if $N$ is a Cartier submodule such that $N_\eta = M_\eta$ for each generic point $\eta$ of $\Supp M$ then $N = M$). An element $x \in R$ is called a \emph{test element} for $(M, \mathcal{C})$ if $D(x) \cap \Supp M \subseteq \Supp M$ is dense and if $M_x$ is $F$-regular.

The \emph{test module} $\tau(M, \mathcal{C})$ of a coherent Cartier module $(M, \mathcal{C})$ is the smallest Cartier submodule of $M$ that generically agrees with $\underline{M}$. Existence of test modules is proven if $R$ is (essentially) of finite type over an $F$-finite field in \cite[Theorem 4.13]{blicklep-etestideale}. Moreover, the test module exists if and only if $(M, \mathcal{C})$ admits a test element. If $x$ is a test element then \[\tau(M, \mathcal{C}) = \sum_{e \geq 1} \mathcal{C}_e x^a \underline{M} \text{ for any } a  > 0.\] Both statements are part of \cite[Theorem 3.11]{blicklep-etestideale}. In particular, if $(M, \mathcal{C})$ is $F$-regular then $\tau(M, \mathcal{C}) = M$.

For us the most important examples of Cartier algebras will be $\mathcal{C} = \langle \kappa \rangle$, where $\kappa: F_\ast M \to M$ is an $R$-linear map and subalgebras of $\mathcal{C}$. In the first case we will also simply write $(M, \kappa)$ for the datum of a Cartier module. Of particular importance is the subalgebra generated in degree $e$ by $\kappa^e f^{\lceil t p^e\rceil}$, where $t \in \mathbb{Q}_{\geq 0}$ and $f \in R$. We will then also write $\tau(M, f^t)$ for the corresponding test module. Varying $t$ one obtains a decreasing filtration $\tau(M, f^t) \subseteq \tau(M, f^s)$ for $t > s$. Moreover, this filtration is \emph{right-continuous} in the sense that for $t$ there is $0 < \eps \ll 1$ such that $\tau(M, f^t) = \tau(M, f^{t + \delta})$ for all $\delta \leq \eps$ (\cite[Proposition 4.16]{blicklep-etestideale}). If $t \geq 0$ then the so-called Brian\c{c}on-Skoda theorem holds, that is, $f \tau(M, f^t) = \tau(M, f^{t+1})$. If $R$ is essentially of finite type over an $F$-finite field then the test module filtration is discrete (\cite[Corollary 4.19]{blicklep-etestideale}) and rational(this follows from discreteness as in \cite[Theorem 3.1]{blicklemustatasmithdiscretenessrationality}).
Finally, the quotient $Gr^t_\tau = \tau(M, f^{t- \eps})/\tau(M, f^{t})$ naturally carries a Cartier structure induced by $\kappa f^{\lceil t (p-1) \rceil}$ (see \cite[Section 4]{staeblertestmodulnvilftrierung}).

A morphism of Cartier modules is a morphism $\varphi: M \to N$ of the underlying $R$-modules such that for every $\kappa_e \in \mathcal{C}_e$ one has $\kappa_e F_\ast^e \varphi = \varphi \kappa_e$.

Suppose that $f: \Spec S \to \Spec R$ is a morphism of $F$-finite rings. Then we obtain a functor of Cartier modules $f_\ast$ as follows. Let $(M, \mathcal{C})$ be a Cartier module on $S$. Consider the $S$-module pushforward $f_\ast M$ and define a Cartier structure as follows: For $\kappa_e \in \mathcal{C}_e$ we define $\kappa_e: F_\ast^e f_\ast M \to M$ as the composition of $F_\ast^e f_\ast M \cong f_\ast F^e_\ast M \xrightarrow{f_\ast \kappa_e} f_\ast M$.

Assume now that $f$ is finite or smooth. Then for finite $f$ we define a functor $f^!$ from $\mathcal{C}$-modules on $R$ to $\mathcal{C}$-modules on $S$ as $N \mapsto \Hom_R(S, N)$ considered as an $S$-module via the first argument\footnote{Note that this deviates from the usual definition since we do not derive, i.e.\ we do not consider $RHom$.}. In this case a homogeneous element $\kappa_e \in \mathcal{C}_e$ acts via $f^! N \ni \varphi \mapsto \kappa_e \circ \varphi \circ F^e$. If $f$ is smooth (of relative dimension $n$) we define $f^! N = \omega_f \otimes f^\ast N$, where $\omega_f = \bigwedge^n \Omega_{S/R}^1$. For the Cartier structure note that a map $\kappa_e: F_\ast^e M \to M$ is equivalent to a map $C_e: M \to F_e^! M$ by duality for finite morphisms (e.g.\ \cite[Theorem 2.17]{blickleboecklecartierfiniteness}). The map $f^! C_e: f^! M \to f^! {F^e}^! M$ composed with the natural isomorphism $f^! {F^e}^! M \cong {F^e}^! f^!M$ then yields a Cartier structure on $f^!M$. We will discuss this in more detail in Section \ref{ShriekFregularity}.

For the rest of this section we restrict to the case that $\mathcal{C} = \langle \kappa \rangle$.
If $M$ is a quasi-coherent Cartier module endowed with a single structural map $\kappa: F_\ast M \to M$ then $M$ is nilpotent if and only if there is $e \geq 0$ such that $\kappa^e M = 0$. We call $M$ \emph{locally nilpotent} if $M$ is the union of nilpotent Cartier submodules. A \emph{(local) nil-isomorphism} is a morphism $\varphi: M \to N$ of Cartier modules for which both $\ker \varphi$ and $\coker \varphi$ are (locally) nilpotent. For a coherent Cartier module the notions of local nilpotence and nilpotence coincide.

Recall that since we assume $R$ to be $F$-finite the structural map $\kappa: F_\ast M \to M$ is equivalent to a map $C: M \to F^! M$ by duality for finite morphisms (cf. \cite[Theorem 2.17]{blickleboecklecartierfiniteness}). Explicitly, $C$ is given by $m \mapsto [r \mapsto \kappa(rm)]$.

The category of nilpotent Cartier modules is a Serre subcategory of the category of coherent Cartier modules. We call the localization of coherent Cartier modules at nilpotent Cartier modules the category of \emph{Cartier crystals}. A Cartier module $M$ is called \emph{minimal} if\footnote{The operations $\underline{\phantom{M}}$ and $\overline{\phantom{M}}$ commute. We will generalize this in Lemma \ref{OverlineFpureCommute}} $M = \overline{\underline{M}}$. The subcategory of minimal Cartier modules is equivalent to Cartier crystals (see \cite[Theorem 3.12]{blickleboecklecartierfiniteness}).

A nil-isomorphisms of Cartier modules becomes an isomorphism in Cartier crystals. In particular, if two compositions of Cartier morphisms coincide, where all but one morphism are nil-isomorphisms, then all of them are nil-isomorphisms. This reasoning will be used frequently in what follows.

Assume that $R$ is regular essentially of finite type over an $F$-finite field $k$ with structural map $f: \Spec R \to \Spec k$. Fix, once and for all, an isomorphism $k \to F^! k$. Then we define $\omega_R = f^! k$. The module $\omega_R$ is invertible and comes equipped with a natural isomorphism $C: \omega_R \to F^! \omega_R$ (via the isomorphism $F^! f^! k \cong f^! F^! k$) which defines a Cartier structure on $\omega_R$. If $R$ is smooth then $\omega_R$ coincides with top dimensional K\"ahler differentials and if $k$ is perfect the isomorphism $C: \omega_R \to F^! \omega_R$ is the adjoint of the Cartier operator. In particular, if $\varphi: \Spec S \to \Spec R$ is a finite or smooth $k$-morphism and $R, S$ are both regular essentially of finite type over $k$ then one has $\varphi^! \omega_R = \omega_S$.

We refer the reader to \cite{blicklep-etestideale} for the theory of test modules and to \cite{blickleboecklecartierfiniteness} and \cite{blickleboecklecartiercrystals} for the theory of Cartier modules and crystals. Also note that everything we have done here easily generalizes to $F$-finite noetherian schemes.

\subsection*{Unit $R[F]$-modules}
If $R$ is smooth over a perfect field $k$ then a unit $R[F]$-module is an $R$-module $\mathcal{M}$ equipped with an isomorphism $\Phi: F^\ast \mathcal{M} \to \mathcal{M}$, where $F: R \to R$ is the absolute Frobenius morphism. A unit $R[F]$-module $\mathcal{M}$ is \emph{locally finitely generated} if there exists a finitely generated $R$-module $M$ and a map $M \to F^\ast M$ such that $\colim_e {F^e}^\ast M \cong \mathcal{M}$ as a unit $R[F]$-modules. If the map $M \to F^\ast M$ is injective then $M$ is called a \emph{root} of $\mathcal{M}$. We will only consider locally finitely generated unit $R[F]$-modules in this article and will refer to them as unit $R[F]$-modules (or unit $F$-modules). A unit $R[F]$-module for which the underlying $R$-module is coherent is called a unit $R[F]$-crystal.

If $f: \Spec S \to \Spec R$ is a morphism of smooth schemes over $k$ then the pull back of a unit $R[F]$-module is simply the pull back of the underlying $R$-module and the structural map is induced by the isomorphism $f^\ast F^\ast \cong F^\ast f^\ast$. For an open immersion $f$ the pushforward $f_+$ coincides with the ordinary pushforward of modules\footnote{Pushforwards are defined for arbitrary morphisms $f$ between smooth schemes but the definition is more involved if $f$ is not \'etale(see \cite[Lemma 4.3.1 and Section 3]{emertonkisinrhunitfcrys}).} and the structural map is given by the adjoint of $f_\ast M \to f_\ast F_\ast M \cong F_\ast f_\ast M$.

Emerton and Kisin establish in \cite{emertonkisinrhunitfcrys} an analogue of the Riemann-Hilbert correspondence in characteristic $p > 0$. Namely they show that for a smooth $k$-scheme $X$ the bounded derived category of unit $R[F]$-modules is equivalent to the bounded derived category of constructible $\mathbb{F}_p$-sheaves on $X_{\acute{e}t}$. Moreover, the heart of the trivial $t$-structure (that is the abelian category of unit $R[F]$-modules) corresponds to perverse constructible sheaves in the sense of Gabber (\cite{gabbertstructures}).

We refer the reader to \cite{emertonkisinrhunitfcrys} and \cite{emertonkisinintrorhunitfcrys} for more details on unit $R[F]$-modules.

\subsection*{Cartier crystals and unit $R[F]$-modules}

We assume that $R$ is smooth over a perfect field $k$. Given a Cartier module $(M, \kappa)$ the adjoint of its structural map is the map $C: M \to F^! M, m \mapsto [r \mapsto \kappa(rm)]$. These maps induced a direct system and an isomorphism $C: \colim_e {F^e}^! M \to F^! \colim_e {F^e}^! M$. Denoting $\colim_e {F^e}^! M$ by $\mathcal{M}$ and tensoring with $\omega_R^{-1}$ we obtain an isomorphism $\mathcal{M} \otimes \omega_R^{-1} \to F^\ast (\mathcal{M} \otimes \omega_R^{-1})$ (cf. \cite[Corollary 5.8]{blickleboecklecartierfiniteness}). Hence, we obtain a functor from Cartier modules to unit $R[F]$-modules and if we restrict this functor to minimal Cartier modules then it induces an equivalence. In particular, we obtain an equivalence of Cartier crystals with unit $R[F]$-modules (\cite[Theorem 5.15]{blickleboecklecartierfiniteness}).

\section{Unit Test modules}
\label{SectionOne}

The goal of this section is to show that for a map of $R$-Cartier modules $\varphi: M \to N$ which is a nil-isomorphism one has $\varphi(\tau(M,f^t)) = \tau(N, f^t)$ for any $f \in R$ and $t \in \mathbb{Q}_{\geq 0}$. Among other things, this will enable us to show that we can define a test module filtration for unit $R[F]$-modules (cf.\ Definition \ref{UnitTestmoduleDefinition}).

\begin{Le}
\label{FPureDifferentAlgebrasCommonAmbientModule}
Let $R$ be an $F$-finite ring, $f \in R$. Let $(N, \kappa), (M, \kappa)$ be coherent Cartier submodules of some Cartier module $(A, \kappa)$. Assume that $\underline{N} = \underline{M}$ then also $\underline{N}_{\mathcal{C}} = \underline{M}_{\mathcal{C}}$, where $\mathcal{C}$ is a Cartier subalgebra of the Cartier algebra generated by $\kappa$.
\end{Le}
\begin{proof}
By \cite[Proposition 3.2 (c)]{blicklep-etestideale} we have $\underline{M}_{\mathcal{C}}, \underline{N}_{\mathcal{C}} \subseteq \underline{M} = \underline{N}$. By definition $\underline{N}_{\mathcal{C}}=  (\mathcal{C}_+)^e N$ for $e \gg 0$. Fix $e \gg 0$ then we get \[ \underline{M}_{\mathcal{C}} = (\mathcal{C}_+)^e \underline{M}_{\mathcal{C}} \subseteq (\mathcal{C}_+)^e \underline{M} = (\mathcal{C}_+)^e \underline{N} \subseteq (\mathcal{C}_+)^e N = \underline{N}_\mathcal{C}.\]
\end{proof}

The following is implicit in \cite{blickleboecklecartierfiniteness} but due to lack of a precise reference we include a proof.

\begin{Le}
\label{AdjointCartierIsNiliso}
Let $R$ be $F$-finite and $(M, \kappa)$ a quasi-coherent Cartier module. Then $C: M \to F^! M$ is a nil-isomorphism.
\end{Le}
\begin{proof}
Note that the Cartier structure on $F^!M$ is given by $\kappa: F_\ast F^! M \to F^! M, \varphi \mapsto C(\varphi(1))$. This is the adjoint of the natural map $F^! C: F^! M \to F^! F^! M$ (cf.\ \cite[Theorem 2.17, Proposition 2.18]{blickleboecklecartierfiniteness}).
We have to show that $\ker C$ and $\coker C$ are both nilpotent. The nilpotence of $\ker C$ is immediate from \cite[proof of Lemma 6.9]{staeblertestmodulnvilftrierung}. For the nilpotence of $\coker C$ observe that given $\varphi \in F^!M$ we have $\kappa(\varphi) = C(\varphi(1)) \in C(M)$ hence $\kappa (F^! M) \subseteq C(M)$.
\end{proof}

\begin{Le}
\label{OverlineFpureCommute}
Let $R$ be $F$-finite and $(M, \kappa)$ be a coherent Cartier module. Let $\mathcal{C}$ be a subalgebra of $\langle \kappa \rangle$. Then $\underline{(\overline{M}^{\mathcal{C}})}_\mathcal{C} = \overline{(\underline{M}_{\mathcal{C}})}^{\mathcal{C}}$.
\end{Le}
\begin{proof}
We omit $\mathcal{C}$ from the notation. We have inclusions $\underline{M}_{\text{nil}} \subseteq M_{\text{nil}}$ and $\underline{M} \subseteq M$ and an equality $\underline{M}_{\text{nil}} = M_{\text{nil}} \cap \underline{M}$. 
We thus obtain an injective map of Cartier modules $\overline{(\underline{M})} \to \overline{M}$ which is a $\mathcal{C}$-nil-isomorphism. Note that $\overline{(\underline{M})}$ is $F$-pure so that the injection factors via $\overline{(\underline{M})} \to \underline{(\overline{M})}$. But since the codomain is $F$-pure this $\mathcal{C}$-nil-isomorphism is an isomorphism.
\end{proof}

\begin{Le}
\label{OverlineSupportFPure}
If $M$ is $F$-pure with respect to $\mathcal{C}$ then $\Supp \overline{M}^\mathcal{C} = \Supp M$.
\end{Le}
\begin{proof}
Since $\overline{M}^\mathcal{C}$ is a quotient of $M$ we only have to show the inclusion from right to left. Let $x \in \Supp M$. Since $F$-purity localizes we have $\mathcal{C}_+ M_x = M_x$. But then $M_x$ is not nilpotent. 
\end{proof}

The following result should be of independent interest.

\begin{Prop}
\label{TestelementOverlineEquivalence}
Let $R$ be $F$-finite and let $(M, \kappa)$ be a coherent Cartier module and let $\mathcal{C}$ be a Cartier subalgebra of $\langle \kappa \rangle$. Then $f$ is a test element for $(M, \mathcal{C})$ if and only if $f$ is a test element for $(\overline{M}, \mathcal{C})$.
\end{Prop}
\begin{proof}
By Lemma \ref{OverlineFpureCommute} we may assume that $M$ is $F$-pure with respect to $\mathcal{C}$. By Lemma \ref{OverlineSupportFPure} the condition on the support for a given element $f$ is vacuous. So we only have to show that $\overline{M}^{\mathcal{C}}_f$ is $F$-regular if and only if $M_f$ is $F$-regular. Recall from \cite[Lemma 2.18]{blicklep-etestideale} that the operation $\overline{\phantom{M}}$ and $F$-purity commute with localization. So we have a short exact sequence $0 \to {M}_{\text{nil}} \to M \to \overline{M} \to 0$ with both $M$ and $\overline{M}$ $F$-pure and $\Supp M = \Supp \overline{M}$ and we have to show that $M$ is $F$-regular if and only if $\overline{M}$ is $F$-regular.

Let $N \subseteq M$ be a Cartier module that agrees with $M$ at generic points of $\Supp M$. Then for such a generic point $\eta$ we have $N_\eta = M_\eta$ and by \cite[Lemma 6.10]{staeblertestmodulnvilftrierung} we get ${N_{\text{nil}}}_\eta = {M_{\text{nil}}}_\eta$. It follows that $\overline{M}_\eta = \overline{N}_\eta$ and thus $\overline{M} = \overline{N}$ by $F$-regularity. Since $M \to \overline{M}$ and $N \to \overline{N}$ are $\mathcal{C}$-nil-isomorphisms we obtain that the inclusion $N \to M$ is one as well. But $M$ is $F$-pure so that this is an isomorphism as desired.

For the other direction let $N \subseteq \overline{M}$ be a Cartier module that agrees with $\overline{M}$ at generic points of $\Supp \overline{M}$. Let $N'$ be the preimage of $N$ under the surjection $\pi: M \to \overline{M}$. This is a Cartier module since $\pi$ is a morphism of Cartier modules. Fix a generic point $\eta$ of $\Supp \overline{M}$ then $N_\eta = \overline{M}_\eta$ and hence $N'_\eta =M_\eta$. By virtue of $M$ being $F$-regular we conclude that $N' = M$. It follows that $N = \overline{M}$.
\end{proof}

\begin{Le}
\label{DoubleOverline}
Let $R$ be $F$-finite and let $(M, \kappa)$ be a coherent Cartier module and let $\mathcal{C}$ be a Cartier subalgebra of $\langle \kappa \rangle$. Then we have an inclusion $M_{\nil,\kappa} \subseteq M_{\nil, \mathcal{C}}$ and $\overline{\overline{M}^\kappa}^\mathcal{C} = \overline{M}^\mathcal{C}$.
\end{Le}
\begin{proof}
The inclusion is immediate since any element of $(\mathcal{C}_+)^e$ may be written as $\kappa^n r$ for some $r \in R$ and $n \in \mathbb{N}$. The second claim follows from the first since the kernel of the surjection $M \to \overline{\overline{M}^\kappa}^\mathcal{C}$ is just $M_{\nil, \kappa} + M_{\nil, \mathcal{C}}$.
\end{proof}

Note that the corresponding result for $\underline{\phantom{M}}$ (i.e.\ $\underline{{\underline{A}_\kappa}}_\mathcal{C} = \underline{A}_\mathcal{C}$ follows from Lemma \ref{FPureDifferentAlgebrasCommonAmbientModule} by setting $N = \underline{A}_\kappa$ and $M = A$.

\begin{Le}
\label{TestelementImplication}
Let $R$ be $F$-finite, $(M, \kappa)$ a coherent Cartier module and let $\mathcal{C}$ be a Cartier subalgebra of $\langle \kappa \rangle$. Then $f$ is a test element for $M$ with respect to $\mathcal{C}$ if and only if $f$ is a test element for $\overline{M}^\kappa$ with respect to $\mathcal{C}$.
\end{Le}
\begin{proof}
By Proposition \ref{TestelementOverlineEquivalence} $f$ is a test element for $\overline{\overline{M}^\kappa}^\mathcal{C}$ if and only if it is one for $\overline{M}^\kappa$. On the other hand since by Lemma \ref{DoubleOverline} we have $\overline{\overline{M}^\kappa}^\mathcal{C} = \overline{M}^\mathcal{C}$ this is, again by Proposition \ref{TestelementOverlineEquivalence}, the case if and only if $f$ is a test element for $M$.
\end{proof}

Note that since $F^!$ is left exact one has $\overline{F^! M} = F^! M$ if $M = \overline{M}$. 
The following Theorem generalizes \cite[Lemma 4.3]{staeblertestmodulnvilftrierung} and is valid for any $F$-finite $R$ provided that test modules do exist. It will allow us to define test module filtrations for unit $F$-modules (cf.\ Definition \ref{UnitTestmoduleDefinition} below).

\begin{Theo}
\label{UnitTestModuleFiltration}
Let $R$ be essentially of finite type over an $F$-finite field. Given a nil-isomorphism of coherent Cartier modules $\varphi: M \to N$ one has $\varphi(\tau(M, f^t)) = \tau(N, f^t)$. Moreover, the test module filtration $\tau(M, f^t)$ induces a decreasing filtration $\tau(\mathcal{M}, f^t)$ for $\mathcal{M} = \colim {F^e}^! M$ and we obtain $\Phi(\tau(\mathcal{M} ,f^t)) = \tau(\mathcal{N}, f^t)$, where $\Phi: \colim {F^e}^! M \to \colim {F^e}^! N = \mathcal{N}$ is the isomorphism induced by $\varphi$.
\end{Theo}
\begin{proof}
Fix $t \in \mathbb{Q}_{\geq 0}$ and denote the algebra generated in degree $e$ by $\kappa f^{\lceil tp^e \rceil}$ by $\mathcal{C}$. The operations $\underline{\phantom{M}}$ and $\overline{\phantom{M}}$ will always be taken with respect to $\mathcal{C}$ and we will omit this from the notation. We claim that $\varphi(\underline{M}) = \underline{N}$. Indeed, since the inclusion $\underline{N} \subseteq N$ is a nil-isomorphism (and similarly for $\underline{M} \subseteq M$) the restriction $\varphi: \underline{M} \to \underline{N}$ is also a nil-isomorphism. But $\underline{N}$ does not admit non-trivial nilpotent quotients by \cite[Corollary 2.14]{blicklep-etestideale}.

Next, we want to argue that $\underline{M}$ and $\underline{N}$ admit a common test element. Observe that $\varphi(\underline{M}_{\nil}) \subseteq \underline{N}_{\nil}$ so that $\varphi$ induces a nil-isomorphism $\overline{\underline{M}} \to \overline{\underline{N}}$ which is in fact an isomorphism (we already know that it is surjective and injectivity follows from \cite[Lemma 2.12]{blicklep-etestideale}\footnote{Again, note that there is a typing error in the reference ``quotients`` should read ``submodules``.}). Let now $c$ be a test element for $\underline{M}$. By Proposition \ref{TestelementOverlineEquivalence} $c$ is then a test element for $\overline{\underline{M}} \cong \overline{\underline{N}}$ and applying \ref{TestelementOverlineEquivalence} once more we obtain that $c$ is also a test element for $\underline{N}$.

From this we deduce that \[\varphi(\tau(M, f^t)) = \sum_{e \geq 1} \varphi(\kappa^e f^{\lceil tp^e \rceil} c \underline{M}) = \sum_{e \geq 1} \kappa^e f^{\lceil tp^e \rceil} c \underline{N} = \tau(N,f^t).\]
For the second statement observe that $C: M \to F^!M$ is a nil-isomorphism (Lemma \ref{AdjointCartierIsNiliso}) and define $\tau(\mathcal{M}, f^t)$ as the image of $\tau(M,f^t)$ under the natural map $M \to \mathcal{M}$.
\end{proof}

\begin{Bem}
If we assume that $\overline{M} = M$, i.e.\ that the maps $M \to {F^e}^! M$ are injective (cf.\ \cite[Lemma 6.9]{staeblertestmodulnvilftrierung}), then the various $\tau(M, f^t)$ are all contained in the image of $M \to \colim {F^e}^!M$. In particular, this filtration is exhaustive if and only if $C: M \to F^!M$ is an isomorphism. This is precisely the case when $M$ corresponds to a locally constant sheaf under the Riemann-Hilbert correspondence.

Also note that, by the same token, if $M$ is $F$-regular, then $F^! M$ is $F$-regular if and only if $C: M \to F^!M$ is surjective.
\end{Bem}

For later use we record the following results:

\begin{Le}
\label{TauofTauEqualsTau}
Let $(M, \kappa)$ be an $F$-pure coherent Cartier module and $f \in R$ an $M$-regular element such that $M_f$ is $F$-regular. Then $\tau(M, f^t) = \tau(\tau(M, f^0), f^t)$ for all $t \geq 0$.
\end{Le}
\begin{proof}
If $t =0$ then the assertion is trivial. We now assume $t > 0$. We have $\tau(M, f^0) \subseteq M$ so that the inclusion from right to left is clearly satisfied. For the other direction denote the algebra generated in degree $e$ by $\kappa^e f^{\lceil t p^e\rceil}$ by $\mathcal{C}$. Note that by \cite[Lemma 4.1]{blicklestaeblerbernsteinsatocartier} we have $\kappa^a f M \subseteq \kappa^{a+1} f M$. We conclude that $\tau(M, f^0) = \kappa^a f M$ for all $a \gg 0$. Fix one such $a$. 

By the first part of the proof of \cite[Lemma 4.3]{blicklestaeblerbernsteinsatocartier} (which only requires $F$-purity and no assumptions on $t$) we have $\mathcal{C}_+ M = \kappa^e f^{\lceil tp^e \rceil} M$ for all $e \gg 0$. In particular, for $e$ sufficiently large and $e \geq a$ we have \[\mathcal{C}_+ M = \kappa^e f^{\lceil tp^e \rceil} M = \kappa^{e-a} \kappa^a f^{\lceil tp^e \rceil} M\subseteq \kappa^a fM\] since $t > 0$ and since $\kappa \kappa^a f M = \kappa^a f M$ by $F$-purity of $\kappa^a f M$.
We conclude that $\mathcal{C}_+^h M \subseteq \mathcal{C}_+^{h-1} \kappa^a f M$. But for $h \gg 0$ we have by definition $\mathcal{C}_+^h M = \underline{M}_{\mathcal{C}}$ and likewise $\mathcal{C}_+^{h-1} \kappa^a f M = \underline{\kappa^a f M}_{\mathcal{C}}$.
\end{proof}

\begin{Prop}
\label{ExponentAndRootTestModule}
Assume that $R$ is essentially of finite type over an $F$-finite field.
Let $(M, \kappa)$ be an $F$-pure coherent Cartier module and $f \in R$ an $M$-regular element such that $M_f$ is $F$-regular. Assume furthermore that $f = g^n$ for some $g \in R$. Then $\tau(M, f^{\frac{a}{n}}) = \tau(M, g^a)$.
\end{Prop}
\begin{proof}
By Lemma \ref{TauofTauEqualsTau} we may replace $M$ by $\tau(M, f^0)$ and therefore assume that $M$ is $F$-regular. Write $t = \frac{a}{n}$. By right-continuity of the test module filtration there is $t' \in \mathbb{Z}[\frac{1}{p}]$ such that $\tau(M, f^t) = \tau(M, f^{t'})$ and $\tau(M, g^{tn}) = \tau(M, g^{t'n})$. Then by \cite[Lemma 4.4]{blicklestaeblerbernsteinsatocartier} we have $\tau(M, f^{t'}) = \kappa^e(f^{t'p^e} M) = \kappa^e(g^{nt'p^e}M) =\tau(M, g^{nt'})$ for all $e \gg 0$. Hence, we get \[\tau(M, f^t) = \tau(M, f^{t'}) = \tau(M, g^{nt'}) = \tau(M, g^{tn}) = \tau(M, g^a)\] as desired.
\end{proof}

\section{Extending the test module filtration}
\label{SectionTestModuleFiltrationExtension}

Consider the following situation: $R$ is an $F$-finite ring and $f$ a non-zero divisor, denote by $j$ the open immersion $D(f) \to \Spec R$ and let $M$ be an $R_f$-Cartier module. Then $j_\ast M$ is a quasi-coherent Cartier module. In this section we will show how to attach a test module filtration to $j_\ast M$ along $f^t$.

\begin{Le}
\label{CoherentModelOfPushforwardExists}
Let $M$ be a Cartier module on $R_f$ and write $j$ for the open immersion $\Spec R_f \to \Spec R$. Then $j_\ast M$ is locally nil-isomorphic to a coherent Cartier module $N \subseteq i_\ast M$. Moreover, $N$ admits no nilpotent submodules if $M$ admits none.
\end{Le}
\begin{proof}
We may factor $j$ as the composition of $\Spec R_f \xrightarrow{i} \mathbb{A}^1_R \xrightarrow{g} \Spec{R}$, where $i$ is induced by $R[T] \mapsto R[T]/(Tf -1)$. By \cite[proof of Theorem 3.2.14]{blickleboecklecartiercrystals} applied to $i_\ast M$ there is a Cartier module $N \subseteq j_\ast M$ which is coherent and locally nil-isomorphic to $j_\ast M$.

For the addendum note that $N$ has no nilpotent submodules if and only if the adjoint of the structural map $N \to F^! N$ is injective (cf.\ \cite[Lemma 6.9]{staeblertestmodulnvilftrierung}). The adjoint structural map for $N$ is induced by $j_\ast M \to F^! j_\ast M$ which in turn is given by $m \mapsto [r \mapsto \kappa(\varphi(r) m )]$ where $\varphi: R \to R_f$. By assumption $M \to F^!M$ is injective, so given $m$ there is $s = \frac{r}{f^n}$ such that $\kappa(sm)$ is non-zero. But then also $f^e \kappa(sm) = \kappa(sf^{pe}m)$ is non-zero for any $e \geq 0$. Choosing $e$ sufficiently large one has $sf^{pe} \in R$. This shows that $j_\ast M \to F^! j_\ast M$ is injective which implies the injectivity of $N \to F^! N$.
\end{proof}

\begin{Bsp}
Note that we cannot expect that $N$ is $F$-pure if $M$ is $F$-pure. Indeed, consider $M = k[x,x^{-1}]$ then $N_n = k[x] x^{-n}$ is locally nil-isomorphic to $j_\ast M$ (cf.\ \cite[Remark 5.10]{staeblertestmodulnvilftrierung} for computations), where $n \geq 1$ but it is $F$-pure only if $n =1$. However, $j^! N_n$ coincides with $j^! j_\ast M = M$. In particular, since $M$ is $F$-regular these modules have a common test element $x$. As Proposition \ref{TestModuleIndependentOfLNiliso} shows this holds in general.
\end{Bsp}

\begin{Prop}
\label{CoherentModelOfPushforwardPulledBack}
Let $R$ be an $F$-finite ring.
Let $M$ be an $F$-pure Cartier module on $R_f$ and write $j$ for the open immersion $\Spec R_f \to \Spec R$. Let $N \subseteq j_\ast M$ be a coherent Cartier module for which the inclusion is a  local nil-isomorphism. Then $j^! N = j^! j_\ast M = M$.
\end{Prop}
\begin{proof}
First of all, note that $j^! = j^\ast$ since $j$ is \'etale. Since $j^!$ preserves (local) nil-isomorphisms (\cite[Lemma 2.2.3]{blickleboecklecartiercrystals}), local nilpotence and nilpotence coincide for coherent Cartier modules and $j^!$ is exact we obtain an inclusion $j^! N \subseteq M$ which is a nil-isomorphism. Hence, the quotient $M/j^!N$ is nilpotent. But $M$ is $F$-pure so does not admit any non-trivial nilpotent quotients (\cite[Corollary 2.14]{blicklep-etestideale}). This shows that the inclusion is surjective.
\end{proof}

\begin{Le}
\label{CoherentModuleOfPushforwardIntersectionFPure}
Let $R$ be an $F$-finite ring.
Let $M$ be a Cartier module on $R_f$ and write $j$ for the open immersion $\Spec R_f \to \Spec R$. Let $N, N' \subseteq j_\ast M$ be coherent Cartier modules for which the inclusion is a local nil-isomorphism. Then $\underline{N'} = \underline{N}$. In fact, $\underline{N} = \bigcap_{A} A$, where the intersection runs over all coherent Cartier submodules $A \subseteq j_\ast M$ for which the inclusion is a nil-isomorphism.
\end{Le}
\begin{proof}
By \cite[Corollary 2.15]{blickleboecklecartierfiniteness} the natural inclusion $\underline{N} \subseteq N$ is a nil-isomorphism. We conclude that $\underline{N} \subseteq j_\ast M$ is a local nil-isomorphism. This shows the inclusion from right to left.

Let now $A$ be any submodule of $j_\ast M$ for which the inclusion is a local nil-isomorphism. By definition of $A \subseteq j_\ast M$ being a local nil-isomorphism there is an ascending sequence of coherent Cartier submodules $M_n$ such that $\bigcup_n M_n = j_\ast M$ and $\kappa^n M_n \subseteq A$. By coherence there is $n$ such that $\underline{N} \subseteq M_n$ and then $\kappa^n \underline{N} \subseteq A$. But since $\underline{N}$ is $F$-pure we have $A \supseteq \kappa^n \underline{N} = \underline{N}$ as desired.
\end{proof}

\begin{Prop}
\label{TestModuleIndependentOfLNiliso}
In the situation of Proposition \ref{CoherentModuleOfPushforwardIntersectionFPure} one has $\tau(N, f^t) = \tau(N', f^t)$ as submodules of $j_\ast M$ for all $t \in \mathbb{Q}_{\geq 0}$. If additionally $(M, \kappa)$ is $F$-regular then $f$ is a test element for $N, N'$.
\end{Prop}
\begin{proof}
We denote the Cartier algebra generated in degree $e$ by $\kappa^e f^{\lceil tp^e\rceil}$ by $\mathcal{C}$.

Combining Lemmata \ref{CoherentModuleOfPushforwardIntersectionFPure} and \ref{FPureDifferentAlgebrasCommonAmbientModule} we obtain that $\underline{N}_\mathcal{C} = \underline{N'}_{\mathcal{C}}$. In particular, we get $\tau(N, f^t)= \tau(N', f^t)$ due to \cite[Proposition 3.2 (a)]{blicklep-etestideale}.

Assume now that $(M, \kappa)$ is $F$-regular. We claim that $f$ is a test element in the sense of \cite[Theorem 3.11]{blicklep-etestideale} for $N$ and $N'$. By \cite[Lemma 4.1]{staeblertestmodulnvilftrierung} and the above one has $\Supp \underline{N} = \Supp \underline{N}_{\mathcal{C}}$. Moreover, $\mathcal{C}_f$ coincides with the $R_f$-Cartier algebra generated by $\kappa$. Hence, it suffices to show that $D(f) \cap \Supp \underline{N} \subseteq \Supp \underline{N}$ is dense. Since multiplication by $f$ is injective on $M$ it is also injective on $N$. We conclude that $f$ is not contained in any minimal prime of $\Supp N$. Hence, $D(f)$ has non-empty intersection with each irreducible component as desired.
\end{proof}

\begin{Def}
\label{DefTestModuleFiltrationQuasicoherentPushforward}
Let $R$ be $F$-finite and let $M$ be a Cartier module on $R_f$, where $j: \Spec R_f \to R$ is the open immersion and let $N \subseteq j_\ast M$ be a coherent Cartier module locally nil-isomorphic to $j_\ast M$. Then we define the \emph{test module $\tau(j_\ast M, f^t)$ of $M$ for $f^t$ with $t \in \mathbb{Q}_{\geq 0}$} as $\tau(N,f^t)$. Moreover, for $t < 0$ we define $\tau(j_\ast M, f^t)$ as $f^{\lfloor t \rfloor} \tau(j_\ast M, f^{\{t\}})$.
\end{Def}

Note that in this way Skoda's theorem is preserved in the sense that $f^n \tau(j_\ast M, f^t) = \tau(j_\ast M, f^{t+n})$.

\begin{Prop}
\label{TestmodulefiltrationExtensionProperties}
Let $R$ be essentially of finite type over an $F$-finite field. Let $M$ be an $F$-regular Cartier module on $R_f$, where $j: \Spec R_f \to \Spec R$ is the open immersion. Then the test module filtration $\tau(j_\ast M, f^t)_{t \in \mathbb{Q}}$ is exhaustive, discrete and right-continuous.
\end{Prop}
\begin{proof}
Discreteness and right-continuity for $t \leq 0$ are clear by construction. For $t \geq 0$ it follows from \cite[Proposition 4.16, Corollary 4.19]{blicklep-etestideale}.

In order to show that it is exhaustive let $N \subseteq j_\ast M$ be coherent, $F$-pure and nil-isomorphic to $j_\ast M$. Then by Proposition \ref{CoherentModelOfPushforwardPulledBack} we have $j^! N = j^! j_\ast M = M$. Since $j^! = j^\ast$ is flat we obtain that $R_f \cdot N = M$. 
By Proposition \ref{TestModuleIndependentOfLNiliso} the element $f$ is a test element for $N$. Moreover, for $t =0$ the algebra $\mathcal{C}$ is generated by $\kappa$ so that $\underline{N}_\mathcal{C} = N$ by $F$-purity. We thus get $\tau(j_\ast M, f^0) = \sum_{e \geq 1} \kappa^e f N = \kappa^e f N$ for $e \gg 0$, where the last equality follows from $\kappa^e f N \supseteq \kappa^e f^p N = \kappa^{e-1} f \kappa N = \kappa^{e-1} f N$. Hence, $f^{-n} \tau(M, f^0) = \kappa^e f^{-np^e +1} N$ which is an exhaustive filtration since $M$ is $F$-pure.
\end{proof}

\begin{Prop}
\label{LNIlIsoInducesIsoinColim}
Let $R$ be an $F$-finite ring and let $(M, \kappa)$, $(N, \kappa)$ be quasi-coherent Cartier modules. Assume that there exists a morphism $\varphi: M \to N$ which is a local nil-isomorphism. Then $\varphi$ induces an isomorphism $\colim_e {F^e}^! M \to \colim_e {F^e}^! N$ of Cartier modules.
\end{Prop}
\begin{proof}
Note that ${F^e}^! \varphi$ is also a local nil-isomorphism. Indeed, by Lemma \ref{AdjointCartierIsNiliso} the morphisms $C^e: N \to {F^e}^! N$ and $C^e: M \to {F^e}^! M$ are nil-isomorphisms. Since $C^e \circ \varphi = {F^e}^! \varphi \circ C^e$ the claim follows.

By the above observation it suffices to show that if $m \in \ker \varphi$ then $C^e(R \cdot m) = 0$ for some $e > 0$ and if $n \in N$ then there exists $c > 0$ such that $C^{c}(R \cdot n) \subseteq \im {F^{c}}^! \varphi$. But $C^e(m) = [r \mapsto \kappa^e(rm)]$ and since $\ker \varphi$ is locally nilpotent there is $e > 0$ such that $\kappa^e(R \cdot m) = 0$. The claim about the image follows similarly.
\end{proof}

\begin{Bem}
We note that Proposition \ref{LNIlIsoInducesIsoinColim} also yields an alternative proof of the first part of \ref{TestModuleIndependentOfLNiliso} using Theorem \ref{UnitTestModuleFiltration} if $\overline{j_\ast M} = j_\ast M$. The diagram
\[ \begin{xy} \xymatrix{\colim_e {F^e}^! N \ar[r] & \colim_e {F^e}^! M \\ N \ar[u] \ar[r] & M \ar[u]} \end{xy}\] commutes. Hence, the image of $\tau(N, f^t)$ coincides with the image of $\tau(N', f^t)$, where $N' \subseteq j_\ast M$ is a local nil-isomorphism. The maps to the colimits are all injective by Lemma \ref{CoherentModelOfPushforwardExists} and the assumption on $j_\ast M$. We  conclude that $\tau(N, f^t) = \tau(N', f^t)$ as subsets of $j_\ast M$.
\end{Bem}

Theorem \ref{UnitTestModuleFiltration} and Proposition \ref{LNIlIsoInducesIsoinColim} allow us to give a definition of test modules for unit Cartier-modules (i.e.\ quasi-coherent Cartier-modules, where $M \to F^!M$ is an isomorphism):

\begin{Def}
\label{LimitTestmodulefiltration}
Let $R$ be essentially of finite type over an $F$-finite field and $f \in R$. If $M$ is a coherent Cartier module or $M = j_\ast N$ for a coherent Cartier module $N$ on $R_f$, where $j: \Spec R_f \to \Spec R$ is the open immersion, then we define the \emph{test module filtration $\tau(\mathcal{M}, f^t)_{t \in \mathbb{Q}_{\geq 0}}$ along $f$ of $\mathcal{N}$} as the image of $\tau(M, f^t)$ in $\mathcal{M} = \colim {F^e}^! M$.
\end{Def}

This then also yields a notion of test module for unit $R[F]$-modules. Recall that for a unit $R[F]$-module $\mathcal{N}$ the module $\mathcal{N} \otimes \omega_R^{-1}$ carries a unit Cartier-module structure (i.e.\ the structural map $C: N \otimes \omega_R^{-1} \to F^! (N \otimes \omega_R^{-1})$ is an isomorphism).

\begin{Def}
\label{UnitTestmoduleDefinition}
Let $R$ be smooth over an $F$-finite field, $f \in R$ and $\mathcal{N}$ a unit $R[F]$-module. Then we define the \emph{test module filtration $\tau(\mathcal{N}, f^t)_{t \in \mathbb{Q}_{\geq 0}}$ for $\mathcal{N}$} as $\tau(\mathcal{N} \otimes \omega^{-1}, f^t) \otimes \omega_R$ using Definition \ref{LimitTestmodulefiltration}.
\end{Def}

Note that one can also define this filtration using roots (we could even drop the injectivity assumption) of unit $R[F]$-modules. Indeed, let $N$ be a root of $\mathcal{N}$. Then $N \otimes \omega_R^{-1}$ is a Cartier module and the image of $\tau(N \otimes \omega_R^{-1}) \otimes \omega_R$ in the limit $\colim_e {F^e}^\ast N = \mathcal{N}$ coincides with $\tau(\mathcal{N}, f^t)$.

\section{$F$-regularity along $f^!$}
\label{ShriekFregularity}
In this section we study the behavior of $F$-regularity (and thus of test modules) along pull backs by the twisted inverse image of $f: \Spec S \to \Spec R$, where $f$ is either finite flat or smooth. In the smooth case $F$-regularity is preserved. In contrast, $F$-regularity is \emph{not} preserved by finite flat morphisms. In fact, we will see that it already fails for a Kummer covering ramified along a smooth divisor. We start with the smooth case.

Recall that if $f: X \to Y$ is a smooth morphism of schemes then $f^! \bullet$ is given by $\omega_f \otimes f^\ast \bullet$, where $\omega_f = \omega_{X/Y}$ is the relative dualizing sheaf (see e.g.\ \cite[Definition before III.2.1]{hartshorneresidues}).
The Cartier structure is induced by applying $f^!$ to the adjoint structural map $ M \to F^! M$ and using the fact that $f^! F_Y^! M \cong F_X^! f^! M$.

In the following we sketch how the isomorphism $f^! F_Y^! \to F_X^! f^!$ is obtained and make the computation explicit in the case $\mathbb{A}^1_R \to \Spec R$.
First, one factors $F_X = F' F_{rel}$ and obtains the following commutative diagram where the square is cartesian.

\[
\begin{xy}
\xymatrix{X \ar[rd]^{F_{rel}} \ar@/^/[rrd]^{F_X} \ar[rdd]^f& &\\
&X' \ar[r]^{F'} \ar[d]^{f'} &X \ar[d]^{f}\\
&Y\ar[r]^{F_Y} & Y}
\end{xy}\]

By \cite[Corollary III.6.4]{hartshorneresidues} one gets an explicit isomorphism $f'^! F_Y^! \to F'^! f^!$. Next, one unwinds \cite[Proposition III.8.4]{hartshorneresidues} to get an isomorphism $f^! \to F_{rel}^! f'^!$. Finally, one uses the fact that $F_X^! = (F' F_{rel})^! \cong  F_{rel}^! F'^!$ which is just a tensor-hom adjunction since both morphisms are finite.

In what follows, we will write $F^!(-)$ for $\Hom(F_\ast R, -)$ even if $F: \Spec R \to \Spec R$ is not flat.

\begin{Le}
\label{CartierStructureShriekPolyRing}
Let $f: \mathbb{A}^1_{\Spec R} \to \Spec R$ be the structural map and $M$ a $\kappa$-module on $R$. Then the Cartier structure for $f^! M = \omega_f \otimes f^\ast M$ is given by 
\[F_\ast (\omega_f \otimes M) \longrightarrow \omega_f \otimes M,\quad r x^n dx \otimes m \longmapsto x^{\frac{n+1}{p}-1} dx \otimes \kappa(rm), \]
where $\kappa: F_\ast M \to M$ is the Cartier structure on $M$ and $x^{\frac{n+1}{p}-1}$ is understood to be zero whenever $\frac{n+1}{p}$ is not an integer.
\end{Le}

In order to keep the proof of this lemma readable we will discuss the relevant isomorphisms separately. 

\begin{Le}
\label{SmoothPullbackL1}
If $f: \Spec S \to \Spec R$ and $f': \Spec T \to \Spec R$ are smooth, $F_{rel}: \Spec S \to \Spec T$ is finite flat and $f = f' \circ F_{rel}$ then one has for any $R$-module $N$ a natural isomorphism $f^!N \to F_{rel}^! f'^! N$ given by \[\omega \otimes s \otimes n \longmapsto [ s' \mapsto \kappa_{rel}(s's \omega) \otimes 1 \otimes n],\] where $\kappa_{rel}$ is the adjoint of the isomorphism $C_{rel}: \omega_f \to F_{rel}^! \omega_{f'}$. 
\end{Le}
\begin{proof} 
We have an isomorphism $f^\ast N \to F_{rel}^\ast f'^\ast N, s \otimes n \mapsto s \otimes 1 \otimes n$. This isomorphism and $C_{rel}$ induce isomorphisms
\[ \omega_f \otimes_S f^\ast N \to F_{rel}^! \omega_{f'} \otimes_S F_{rel}^\ast f'^\ast N \to \Hom_T(S, \omega_{f'} \otimes_T f'^\ast N) =  F_{rel}^! f'^! N,\]
\[\omega \otimes s \otimes n \longmapsto C_{rel}(s \omega) \otimes 1 \otimes 1 \otimes n \longmapsto [ s' \mapsto \kappa_{rel}(s's \omega) \otimes 1 \otimes n].  \]
Here the last isomorphism is due to the fact that $F_{rel}$ is flat so that $S$ is a locally free $T$-module.
\end{proof}

Note that one has an isomorphism $F'^\ast \omega_f \cong \omega_{f'}$ by base change of K\"ahler differentials. We may identify $F'^\ast \omega_f$ and $\omega_{f'}$ via this isomorphism (cf.\ \cite[Theorem 3.6.1]{conradduality}).

\begin{Le}
\label{SmoothPullbackL2}
If $f: \Spec S \to \Spec R$ is smooth and $F: \Spec T \to \Spec R$ is finite then for any $R$-module $N$ one has a natural isomorphism
\[ F'^\ast \omega_f \otimes_{S \otimes_R T} f'^\ast \Hom_R(T, N) = f'^! F^! N \longrightarrow F'^! f^! N = \Hom_S(S \otimes_R T, \omega_f \otimes_S f^\ast N)\]
\[ a \otimes \omega \otimes a' \otimes \varphi \longmapsto aa' \cdot [ s \otimes t \mapsto \omega \otimes s \otimes \varphi(t) ],\]
where $f'$ and $F'$ are the base changes of $f$ and $F$ and where we identify $\omega_{f'} = F'^\ast \omega_f$.
\end{Le}
\begin{proof}
Let us denote the ring $S \otimes_R T$ by $A$. Then the left hand side is given by $(\omega_{f} \otimes_S A) \otimes_A (A \otimes_T \Hom_R(T, N))$ which is isomorphic to $(\omega_{f} \otimes_S A) \otimes_A( S \otimes_R \Hom_R(T, N)) = (\omega_{f} \otimes_S A) \otimes_A \Hom_S(A, f^\ast N)$ by flat base change. 
Finally, one has an isomorphism \[\omega_f \otimes_S \Hom_S(A, f^\ast N) \to \Hom_S(A, \omega_f \otimes_S f^\ast N), \omega \otimes \psi \mapsto [s \mapsto \omega \otimes \psi(s) ]\] since $\omega_f$ is invertible.
\end{proof}

\begin{proof}[proof of \ref{CartierStructureShriekPolyRing}]
Note that the relative Frobenius $R[x] \to R[x], x \mapsto x^p$ is flat and recall that the relative Cartier operator $\kappa_{rel}$ is given by sending $r x^{m} dx$ to $r \otimes x^{\frac{m+1}{p}-1} dx$ (see e.g.\ \cite[Theorem 7.2]{katznilpotent}).

The adjoint of the relative Cartier operator $\kappa_{rel}: {F_{rel}}_\ast \omega_f \to F'^\ast \omega_f$ is the isomorphism \[\omega_f \to F^!_{rel} F'^\ast \omega_f = \Hom_{R[x]'}({F_{rel}}_\ast R[x], F'^\ast \omega_f), \omega \mapsto [rx^n \mapsto \kappa_{rel}(rx^n \omega)].\]
Once again we identify $F'^\ast \omega_f$ and $\omega_{f'}$ via the base change isomorphism.

With this identification Lemma \ref{SmoothPullbackL1} yields an isomorphism \[\Sigma: f^! F^! M \longrightarrow F_{rel}^! f'^! F^!M,\quad \omega \otimes r \otimes \varphi \longmapsto [r' \mapsto  \kappa_{rel}(r'r\omega) \otimes 1 \otimes \varphi].\]
 Next, by Lemma \ref{SmoothPullbackL2}, we have an isomorphism $f'^! F^! M \to {F'}^! f^! M$. Applying $F_{rel}^!$ to this isomorphism and composing with $\Sigma$ we obtain the isomorphism 
\[ \Xi: \omega_f \otimes_S f^\ast \Hom_R(F_\ast R, M) = f^! F^! M \longrightarrow F_{rel}^! F'^! f^! M \]
\[x^n dx \otimes 1 \otimes \varphi \longmapsto[[rx^m \mapsto [s \otimes t \mapsto  x^{\frac{m+n+1}{p}-1} dx \otimes s \otimes \varphi(rt)]].\]

By tensor-hom adjunction we obtain a natural isomorphism $F_{rel}^! F'^! f^! M \to F^! f^! M, \psi \mapsto [a \otimes b \mapsto \psi(a)(b)]$. Composing this isomorphism with the isomorphism $\Xi$ above we obtain
\[\Lambda: f^! F^! M \longrightarrow F^! f^! M, x^n dx \otimes 1 \otimes \varphi \longmapsto [rx^m \mapsto x^{\frac{n+m+1}{p}-1} dx \otimes 1 \otimes \varphi(r)].\]

Finally, we consider the adjoint Cartier structure $C: M \to F^!M$ and apply $f^!$. This yields $\id_{\omega_f} \otimes \id_{R[x]} \otimes C: f^! M \to f^! F^! M$. Composing with $\Lambda$ we get a map $f^!M \to F^! f^! M$ and taking its adjoint yields
\[F_\ast f^! M \longrightarrow f^! M, x^n dx \otimes 1 \otimes m \longmapsto x^{\frac{n+1}{p}-1} dx \otimes 1 \otimes \kappa(m). \] Making the usual identification of $M \otimes_R S \otimes_S \omega_f \cong M \otimes_R \omega_f$ we obtain the claim.
\end{proof}

\begin{Le}
\label{SmoothShriekPullbackComposition}
Let $f:\Spec S \to \Spec R$ be a smooth morphism such that $f = g \circ \varphi$, where $g: \mathbb{A}^n_R \to \Spec R$ is the structural morphism and $\varphi$ is \'etale. Let $(M, \kappa)$ be a Cartier module on $R$. Then the natural isomorphism $f^! M \to \varphi^! g^! M$ is Cartier linear.
\end{Le}
\begin{proof}
We restrict to the case $n =1$ to keep notations simple. Since $\varphi$ is \'etale we have an isomorphism $F'^\ast S \to {F_{rel}}_\ast S, r \otimes s \mapsto rs^p$, that is, any $s \in S$ may be written as a sum $\sum_i r_i t_i^p$ for suitable $r_i \in R[x],  t_i \in S$. The inverse of this isomorphism is of course the relative Cartier operator $\kappa_{rel}: {F_{rel}}_\ast S \to F'^\ast S$. Since $\omega_f = S dx$ the relative Cartier operator for $f$ is given by \[\kappa_{rel}: {F_{rel}}_\ast \omega_f \to F'^\ast \omega_f,\, rx^a s^p dx \mapsto rs \otimes x^{\frac{a + 1 }{p}-1} dx.\]
Using this formula for the relative Cartier operator it follows, just as in the proof of Lemma \ref{CartierStructureShriekPolyRing}, that the Cartier structure for $f^!M$ is given by \[F_\ast (M \otimes_R \omega_f) \longrightarrow M \otimes_R \omega_f,\, m \otimes rs^p x^a dx \longmapsto  \kappa(rm) \otimes s x^{\frac{a+1}{p} -1}.\]

Using Lemma \ref{CartierStructureShriekPolyRing} and the fact that for a Cartier module $(N, \kappa)$ the Cartier structure on $\varphi^!N$ is given by the natural composition (cf.\ \cite[Lemma 2.2.1, Definition 2.2.2]{blickleboecklecartiercrystals}) \[F_\ast(N \otimes S) \xrightarrow{\cong} F_\ast N \otimes S \xrightarrow{\kappa \otimes \id} N \otimes S\] one verifies that the Cartier structure on $\varphi^! g^! M$ is given by \[F_\ast (M \otimes \omega_g) \otimes S \longrightarrow (M \otimes \omega_g) \otimes S, m \otimes rx^a dx \otimes s^p \mapsto rx^{\frac{a+1}{p} -1}dx \otimes m \otimes s.\]

The isomorphism $\varphi^! g^! M \to f^! M$ is given by the usual identification of $\varphi^\ast g^\ast M$ with $f^\ast M$ and by $\omega_\varphi \otimes_S S \otimes_{R[x]} \omega_g \to \omega_f, s \otimes 1 \otimes r dx \mapsto rsdx$ (cf.\ \cite[Proposition III.2.2]{hartshorneresidues} (also note that no sign error occurs in our situation -- \cite[Section 2.2]{conradduality})). One now readily observes that the claim follows.
\end{proof}

For the following lemma recall that if $i: \Spec B \to \Spec A$ is a finite map and $(M, \kappa)$ a Cartier module over $A$ then $i^! M = \Hom_A(B, M)$ has an induced Cartier structure given by $F_\ast i^! M \to i^! M, \varphi \mapsto \kappa \circ \varphi \circ F$.

\begin{Le}
\label{ClosedImmersionSmoothShriekBCCompatible}
Let $f: \Spec S \to \Spec R$ be a smooth morphism and $i: \Spec R/I \to \Spec R$ a closed immersion. Then we have a natural isomorphism of functors $i'_\ast f'^! \to f^! i_\ast$ that is compatible with Cartier structures, where $i'$ is the base change of $i$ and similarly for $f'$.  In particular, it induces an isomorphism of Cartier crystals.
\end{Le}
\begin{proof}
We write $A$ for $S/IS$. We have a natural isomorphism of $R$-modules $\id \to i^! i_\ast$ given by \[M \to \Hom_R(R/I, i_\ast M), m \mapsto [1 \mapsto m].\]

Next, we apply $f'^!$ to this isomorphism and use the fact that $f'^! i^! = (if')^! = (fi')^! = i'^! f^!$ to get an isomorphism $f'^! \to i'^! f^! i_\ast$. Explicitly, this isomorphism is given by the following composition \begin{multline*}f'^! M = \omega_{f'} \otimes_{A} A \otimes_{R/I} M  \to \omega_{f'} \otimes_{A} S \otimes_R R/I \otimes_{R/I} \Hom_R(R/I, i_\ast M) \to\\ \omega_{f'} \otimes_{A} \Hom_S(A, f^\ast i_\ast M) = \omega_{f'} \otimes_A i'^! f^\ast i_\ast M \to i'^!(f^\ast i_\ast M \otimes \omega_f) = i'^! f^! i_\ast M, \end{multline*}
where the first isomorphism is induced by the natural map above, the second isomorphism is flat base change and the final isomorphism is due to the fact that $i'^\ast \omega_f \cong \omega_{f'}$ via the usual base change morphism of (top-dimensional) K\"ahler differentials and the natural isomorphism $i'^\ast \omega_f \otimes i'^! f^\ast i_\ast M \to i'^!f^! i_\ast M$.

Now we use adjunction of $i'_\ast$ and $i'^!$ (cf.\ \cite[Theorem 2.17]{blickleboecklecartierfiniteness}) to get a morphism
\begin{align*}i'_\ast f'^! M = i'_\ast(\omega_f \otimes_S (S \otimes_R R/I) \otimes_{R/I} M) &\longrightarrow f^! i_\ast M = \omega_f \otimes_S S \otimes_R i_\ast M\\ \omega \otimes s \otimes r \otimes m &\longmapsto \omega \otimes s \otimes rm\end{align*} which clearly is an isomorphism.

Next we verify that this is compatible with Cartier structures. By \cite[Proposition 3.3.23]{blickleboecklecartiercrystals} the adjunction between $i'_\ast$ and $i'^!$ is compatible with Cartier structures. Likewise the counit $i^! i_\ast \to id$ is a morphism of Cartier modules. We therefore only have to check that the isomorphism $f'^! i^! \to i'^! f^!$ of Lemma \ref{SmoothPullbackL2} is compatible with Cartier structures. An application of Lemma \ref{SmoothShriekPullbackComposition} (and working locally) shows that it is sufficient to consider the case where $f$ is \'etale and the case where $f: \mathbb{A}^n_R \to \Spec R$ is the structural map separately (again we restrict to $n =1$ to simplify notation).

In other words, we have to verify that given a Cartier module $(M, \kappa)$ the square
\[\begin{xy}
   \xymatrix{F_\ast f'^! i^! M \ar[r]^\kappa \ar[d] &f'^! i^!M \ar[d]\\
 F_\ast i'^! f^! M \ar[r]^\kappa & i'^! f^! M}
  \end{xy}
\]
commutes. 

We begin with the case that $f: \Spec R[x] \to \Spec R$ is the structural map. Using the usual identification of $\omega_{f'} \cong \omega_f \otimes_{R[x]} R/I[x]$ we can write a generator of $F_\ast f'^! i^! M$ as $x^n dx \otimes 1 \otimes \bar{s} x^j \otimes \varphi \in \omega_f \otimes R/I[x] \otimes R/I[x] \otimes \Hom_R(R/I, M)$. By the natural isomorphism constructed above this element is mapped to $[\bar{r} x^m \mapsto x^n dx \otimes x^{m+j} \otimes \varphi(\overline{rs})]$. Applying $\kappa$ we obtain $[\bar{r} x^m \mapsto \kappa(x^n dx \otimes x^{pm+j} \otimes \varphi(\overline{r^ps})]$ and by Lemma \ref{CartierStructureShriekPolyRing} this evaluates to $[\bar{r} x^m \mapsto x^m x^{\frac{n+1+j}{p} -1} dx \otimes \kappa(\varphi(\overline{r^ps}))]$.

Going the other way around in the diagram we obtain that $x^n dx \otimes 1 \otimes \bar{s} x^j \otimes \varphi$ is mapped via $\kappa$ to $x^{\frac{n+j+1}{p} -1} dx \otimes 1 \otimes 1 \otimes (\kappa  \varphi  \mu_{\bar{s}} F)$, where $\mu_{\bar{s}}$ denotes the map which is multiplication by $\bar{s}$. Now we apply the isomorphism of Lemma \ref{SmoothPullbackL2} to obtain $[\bar{r} x^m \mapsto x^m x^{\frac{n+j+1}{p} -1} dx \otimes \kappa(\varphi(\overline{sr^p}))]$ as claimed.

Next, assume that  $f: \Spec S \to \Spec R$ is \'etale. A generator element of $F_\ast f^! i^! M = F_\ast \Hom_R(R/I, M) \otimes_{R/I} (R/I \otimes_R S)$ is of the form $\alpha \otimes r \otimes s^p$. The left vertical arrow maps this to $[s' \mapsto \alpha(r) \otimes s's^p] \in F_\ast \Hom_S(S/IS, M \otimes_R S)$. This in turn is mapped to $[s' \mapsto \kappa(\alpha(r)) \otimes s s'] \in \Hom_S(S/IS, M \otimes_R S)$ via the Cartier structure. One easily verifies that going the other way in the diagram yields the same result.
\end{proof}

\begin{Le}
\label{FRegularPreservedShriekPoylnomialRing}
Let $f: \mathbb{A}^1_{\Spec R} \to \Spec R$ and $M$ an $F$-regular Cartier module on $\Spec R$ where $R$ is $F$-finite. Then $f^! M$ is $F$-regular.
\end{Le}
\begin{proof}
We identify $f^! M$ with $\omega_f \otimes_R M = R[x] dx \otimes_R M$. If we denote the underlying Cartier algebra of $M$ by $\mathcal{C}$ then the Cartier structure on $f^!M$ is given by $r x^n dx \otimes m \mapsto x^{\frac{n+1}{p^e} -1} dx \otimes \kappa(rm)$, where $x^{\frac{n+1}{p^e} -1}$ is understood to be zero whenever $\frac{n+1}{p^e}$ is not an integer and $\kappa \in \mathcal{C}_e$. By abuse of notation we denote this map by $\kappa^e_f \otimes \kappa$. We denote the Cartier algebra acting on $f^! M$ by $\mathcal{C}'$.

It readily follows that $f^! M$ is $F$-pure since $\kappa_f$ is surjective and $\mathcal{C}_+ M = M$ by assumption. By Lemma \ref{ClosedImmersionSmoothShriekBCCompatible} we may assume that $\Supp M = \Spec R$, where $R$ is reduced by $F$-purity. By flatness we have $R^\circ \subseteq R[x]^\circ$. In order to check $F$-regularity for $f^! M$ it is thus sufficient to show that $(f^! M)_c$ is $F$-regular for some $c \in R^\circ$. Indeed, by \cite[Theorem 3.11]{blicklep-etestideale} one then has \[\tau(f^! M) = \sum_{e \geq 1} \mathcal{C}'_e c f^!M = f^!M\] since $M$ is $F$-regular and $c \in R$.

Since $R$ is excellent we may choose $c$ such that $R_c$ is regular (and replace $R$ by $R_c$). By localizing further we reduce to the case that $R$ is an integral domain. Let $a = \sum_{i=0}^n a_i x^i$ be in $R[x]^\circ$ with $a_n \neq 0$ and hence $a_n \in R^\circ$. Again we use \cite[Theorem 3.11]{blicklep-etestideale} and thus need to show that $\sum_{e \geq 1} \mathcal{C}'_e a f^! M =  f^! M$. 

Choose generators $m_1, \ldots, m_l$ of $M$. Then for each $1 \leq u \leq l$ there are, by $F$-regularity of $M$, elements $C_{j_{iu}} \in \mathcal{C}_{j_{iu}}$ and elements $m_{j_{iu}} \in M$ such that $\sum_{i=1}^{t_u} C_{j_{iu}}(a_n m_{j_{iu}}) = m_u$. Furthermore, we may assume that all $j_{iu} \geq n$ for all $i,u$ since by $F$-purity $\mathcal{C}_+ M = M$. Now choose natural numbers $s_{j_{iu}}$ such that $\kappa^{j_{iu}}_f(x^{n+s_{j_{iu}}}) = 1$ for $i= 1, \ldots, t$ (and thus $\kappa_f^{j_{iu}}(x^k) = 0$ for $k < n + s_{j_{iu}}$). Then \[\sum_{i=1}^{t_u} \kappa_f^{j_{iu}} \otimes C_{j_{iu}}(a \cdot x^{s_{j_{iu}}} \otimes m_{{j_{iu}}})  = \sum_{i=1}^{t_u} \kappa_f^{j_{iu}} (x^{ n +s_{j_{iu}}}) \otimes C_{j_{iu}} (a_n m_{{j_{iu}}}) = 1 \otimes m_u.\]
\end{proof}

\begin{Theo}
\label{FRegularPreservedShriekSmooth}
Let $f: X \to Y$ be a smooth morphism of schemes essentially of finite type over an $F$-finite field and $\mathfrak{a} \subseteq \mathcal{O}_X$ an ideal sheaf, $t \in \mathbb{Q}$ and $M$ a non-degenerate $F$-regular $\kappa^{\mathfrak{a}^t}$-module. Then $f^! M$ is $F$-regular.
\end{Theo}
\begin{proof}
The problem is local on $X$ so that we may replace $X$ by an open affine $\Spec S$ such that $f$ factors as $\Spec S \xrightarrow{\varphi} \mathbb{A}^n_R \to \Spec R$, where $\varphi$ is \'etale and $\Spec R \subseteq Y$ is open affine. By \cite[Theorem 6.15]{staeblertestmodulnvilftrierung} $\varphi^!$ preserves $F$-regularity of $M$. So the claim follows from Lemmata \ref{SmoothShriekPullbackComposition} and \ref{FRegularPreservedShriekPoylnomialRing}.
\end{proof}

\begin{Ko}
\label{SmoothShriekTestModule}
Let $f: \Spec S \to \Spec R$ be a smooth morphism of rings essentially of finite type over an $F$-finite field and let $\mathfrak{a} \subseteq R$ be an ideal. Then for a non-degenerate Cartier module $(M,\kappa)$ one has $f^!(\tau(M, \mathfrak{a}^t)) = \tau(f^! M, (\mathfrak{a}S)^t)$ for all $t \in \mathbb{Q}_{\geq 0}$.
\end{Ko}
\begin{proof}
First of all, note that by \cite[Theorem 4.13]{blicklep-etestideale} both $\tau(f^! M, \mathfrak{a}^t)$ and $\tau(M, \mathfrak{a}^t)$ exist. By exactness of $f^!M$ we have an inclusion $f^! \tau(M, \mathfrak{a}^t) \subseteq f^! M$. We may replace $M$ by $\underline{M}$ and make a base change such that $\Supp M = \Spec R$.

We claim that all generic points of $\Spec S$ are mapped to generic points of $\Spec R$. By a base change we may assume that $\Spec R$ is irreducible. By \cite[Proposition 2.3.4 (iii)]{EGAIV-2} all irreducible components of $\Spec S$ dominate $\Spec R$ from which one easily deduces the claim. 
We conclude that $f^!M$ and $f^! \tau(M, \mathfrak{a}^t)$ agree at generic points of $\Supp f^! M = f^{-1}(\Supp M)$. Since $\tau(f^!M, \mathfrak{a}^t)$ is minimal with this property we obtain an inclusion $\tau(f^!M, \mathfrak{a}^t) \subseteq f^! \tau(M, \mathfrak{a}^t)$. By the previous observation and $F$-regularity of the latter this is an equality.
\end{proof}

\begin{Ko}
Let $f: \Spec S \to \Spec R$ be a smooth morphism of rings essentially of finite type over an $F$-finite field and let $\mathfrak{a} \subseteq R$ be an ideal. Then for a non-degenerate Cartier module $(M,\kappa)$ one has $f^! Gr_\tau^t(M) \cong Gr_\tau(f^!M)$ as Cartier modules or crystals.
\end{Ko}
\begin{proof}
As $f^!$ is exact this follows from Corollary \ref{SmoothShriekTestModule}
\end{proof}

We now present a class of examples that shows that $f^!$, for $f$ finite flat, does not preserve $F$-regularity, not even on the level of Cartier crystals. The argument roughly is as follows. For (say) $R = k[x]$ the Cartier module $(\omega_R, \kappa x^s)$ is $F$-regular if and only if $s < p -1$. If $f: \Spec k[t] \to \Spec k[x]$, where $t^{p-1} = x$ then $f^! \omega_R$ with the Cartier structure induced by $\kappa$ is isomorphic to $\omega_{k[t]}$ with its canonical Cartier structure. Hence, if the Cartier structure is given by $\kappa x^s = \kappa t^{s (p-1)}$ then $f^! \omega_R$ is not $F$-regular.

As usual this is a local issue so that we may immediately restrict to the situation where $f: \Spec S \to \Spec R$ is a morphism of affine schemes. We start with some preparatory results.

\begin{Prop}
\label{PushPullStuff}
 Let $f: \Spec S \to \Spec R$ be a finite morphism and let $j: \Spec R_x \to \Spec R$ be the open immersion induced by localization. Then for any $R_x$-module $M$ one has an isomorphism $f^! j_\ast M \to j'_\ast f'^! M$, where $j'$ and $f'$ are the base changes of $j$ and $f$. If $M$ is a Cartier module then then the natural map is an isomorphism of Cartier modules.
\end{Prop}
\begin{proof}
One has an isomorphism $j'^! f^! j_\ast M  \to f'^! j^! j_\ast M$. Since $j, j'$ are \'etale $j'^! = j'^\ast$ and $j^! = j^\ast$, so by adjunction of $j'^\ast$ and $j'_\ast$ we get an induced morphism $f^! j_\ast M \to j'_\ast f'^! j^\ast j_\ast M$. After the identification $j^\ast j_\ast M = M$ this will be our desired isomorphism.

First of all, the isomorphism \[j'^\ast f^! j_\ast M = \Hom_R(S, j_\ast M) \otimes S_x \longrightarrow f'^! j^\ast j_\ast M = \Hom_{R_x}(S_x, j^\ast j_\ast M)\] is given by (note that $S_x = S \otimes_R R_x$) $\varphi \otimes \frac{s}{x^n} \mapsto \varphi \mu_s \otimes \mu_{\frac{1}{x^n}}$, where $\mu_{a}$ denotes multiplication by $a$. Hence, adjunction yields the map \[\Hom_R(S, j_\ast M) \longrightarrow j'_\ast \Hom_{R_x}(S_x, j^\ast j_\ast M), \quad \varphi \longmapsto \varphi \otimes \id_{R_x}.\]
This is map is injective by flatness of $j$ and it is clearly surjective hence an isomorphism.

We now verify that this isomorphism is compatible with Cartier structures. By \cite[Proposition 3.3.15]{blickleboecklecartiercrystals} the adjunction $j^\ast j_\ast \to id$ is an isomorphism of Cartier modules. We are therefore reduced to showing that $f^!j_\ast M \to j'_\ast f'^! M, \varphi \mapsto [ \frac{s}{x^n} \mapsto \frac{\varphi(s)}{x^n}]$ is compatible with the Cartier structures. The Cartier structure on $f^!j_\ast M$ is given by $\varphi \mapsto \kappa \circ \varphi \circ F_S$, where $\kappa: F_\ast M \to M$. The one on $j'_\ast f'^! M$ is given by $\varphi \mapsto \kappa \circ \varphi \circ F_{S_f}$ and one readily verifies that the relevant diagram commutes. 
\end{proof}

\begin{Prop}
\label{PushforwardConstantSheafLNil}
Let $R$ be $F$-finite, $x \in R$ and $j: \Spec R_x \to \Spec R$ be the open immersion induced by localization. If $(M, \kappa)$ is a Cartier module, then $j_\ast j^! M = j_\ast (M \otimes R_x)$ is locally nil-isomorphic to $M \otimes R \cdot \frac{1}{x}$. Moreover, if $R$ is smooth over an $F$-finite field $k$ and $x$ is a smooth element and $M = \omega_R$ then $\omega_R \otimes R\cdot \frac{1}{x}$ is $F$-pure but not $F$-regular and there are no non-zero nilpotent submodules.
\end{Prop}
\begin{proof}
First of all, note that if $m \in M$ then \[\kappa(m \otimes \frac{1}{x^n}) = \kappa(m \otimes \frac{x^{n(p-1)}}{x^{np}}) = \kappa(x^{n(p-1)} m \otimes \frac{1}{x^{np}}) = \kappa(x^{n(p-1)} m) \otimes \frac{1}{x^n} \in M \otimes R \cdot \frac{1}{x^n}\] since $(F_\ast M) \otimes R_x \to F_\ast(M \otimes R_x), m \otimes a \mapsto m \otimes a^p$ is an isomorphism (cf.\ \cite[Lemma 2.2.1]{blickleboecklecartiercrystals}). We conclude that $M \otimes R \cdot \frac{1}{x^n}$ is a Cartier submodule of $j_\ast j^! M$. 
Next, we write $j_\ast j^! M$ as the union of the $M_n = M \otimes R \cdot \frac{1}{x^n}$ for $n \in \mathbb{N}$. We have to show that there is $e_n$ such that $\kappa^{e_n} M_n \subseteq M \otimes R \cdot \frac{1}{x}$. Choose $e_n$ such that $p^{e_n} \geq n$. Then for $m \in M$ we have \[\kappa^{e_n}(m \otimes \frac{1}{x^n}) = \kappa^{e_n}(m \otimes \frac{x^{p^{e_n} -n}}{x^{p^{e_n}}}) = \kappa^{e_n}(x^{p^{e_n} -n} m) \otimes \frac{1}{x} \in M \otimes R \cdot \frac{1}{x}.\]

Assume now that $R$, $x$ are smooth and $M = \omega_R$. We have to show that $\kappa: F_\ast (\omega_R \otimes R \cdot \frac{1}{x}) \to \omega_R \otimes R \cdot \frac{1}{x}$ is surjective. This is local on $R$ so that we may restrict to stalks $R_P$, where we have a local system of parameters $x_1, \ldots, x_n$. If $x$ is invertible in $R_P$ then the map is clearly surjective. 
If $x$ is not invertible then its part of a sequence of parameters by smoothness. Indeed, let $x_1, \ldots, x_n$ be a sequence of parameters, then $dx = \sum_i \frac{\partial x}{\partial x_i} dx_i$ in $\Omega_R^1$\footnote{Here we use that locally $R$ is \'etale over $k[x_1, \ldots, x_n]$}. As $x$ is smooth some partial derivative is a unit (say, after reordering, the first one). Then the transformation matrix sending $dx_1$ to $dx$ is invertible. In particular, in terms of the basis $\delta = dx \wedge dx_2 \wedge \cdots \wedge dx_n$ we obtain \[\kappa(x_2^{p-1} \cdots x_n^{p-1} \delta \otimes \frac{1}{x}) = \kappa( x^{p-1} x_2^{p-1} \cdots x_n^{p-1} \delta \otimes \frac{1}{x^p}) = (\delta \otimes \frac{1}{x}).\] The claim that $\omega_R \otimes R \cdot \frac{1}{x}$ admits no non-zero nilpotent submodules is again local so that we may assume that $\omega_{R}$ is free generated by $\delta$. We will in fact show that $\omega_R \otimes R \cdot \frac{1}{x}$ endowed with Cartier structure $\kappa x^s$ for any $s \geq 0$ admits no nilpotent submodules.

If $N \neq 0$ is any Cartier submodule and $r \delta \otimes \frac{1}{x} \in N$ a non-zero element then \begin{align*}(\kappa x^s)^e((r^{p^e-1} x^t x_2^{p^e-1} \cdots x_n^{p^e-1}) \cdot r \delta \otimes \frac{1}{x}) &= \kappa^e( x^{s\frac{p^e-1}{p-1} + t} x_2^{p^e-1} \cdots x_n^{p^e-1} r^{p^e} \delta \otimes \frac{1}{x})\\ &= r x^{p^{a-e}} \delta \otimes \frac{1}{x} \neq 0\end{align*} for $e \geq 1$, where $t = p^e -1 - s\frac{p^e-1}{p-1} + p^a$ for some $a \gg 0$, so that $N$ is not nilpotent. 
In order, to see that $\omega_R \otimes R \cdot \frac{1}{x}$ is not $F$-regular it suffices to note that $\omega_R \otimes R$ is a Cartier submodule that generically agrees with $\omega_R \otimes R \cdot \frac{1}{x}$. 
\end{proof}

We can also prove a slight variant of this result by twisting the Cartier structure with premultiplication by a power of $x$
\begin{Le}
\label{PushforwardLocallyConstantSheafLNil}
Let $j: \Spec R_f \to \Spec R$ be the open immersion induced by localization. Then $j_\ast \omega_{R_f} = j_\ast \omega_R \otimes R_f$ endowed with the Cartier structure $\kappa f^s$ is locally nil-isomorphic to $\omega_R$ for $s \geq 1$. If $R$ is smooth over an $F$-finite field, $x$ is a smooth element and $0 \leq s \leq p-1$ then $(\omega_R, \kappa x^s)$ is $F$-pure. Moreover, $(\omega_R, \kappa x^s)$ is $F$-regular if and only if $s \leq p-2$.
\end{Le}
\begin{proof}
Everything except the $F$-regularity follows just as in Proposition \ref{PushforwardConstantSheafLNil}. For the $F$-regularity note that (working locally) we have an \'etale morphism $\varphi: \Spec R \to \Spec S = \Spec k[x_1, \ldots, x_n]$, where $x = x_1$, and $\omega_R = \varphi^! \omega_S$. By \cite[Theorem 6.15]{staeblertestmodulnvilftrierung} $\omega_R$ is $F$-regular if $\omega_S$ endowed with Cartier structure $\kappa x_1^s$ is $F$-regular.\footnote{Note that since we do not assume that our base field is perfect the Cartier structure is given as in Lemma \ref{CartierStructureShriekPolyRing} above. In particular, everything depends on the choice of an isomorphism $k \to F^! k$.} Write $\delta = dx_1 \wedge \ldots \wedge dx_n$.

We have to show that for $g \in k[x_1, \ldots, x_n]$ a non-zero polynomial there is $e \geq 1$ such that $(\kappa x_1^s)^e(Rg \delta) = R \delta$ (cf.\ \cite[Proposition 3.7]{blicklep-etestideale}). Note that $(\kappa x_1^s)^e = \kappa^e x_1^{s \frac{p^e -1}{p-1}}$. Denote the largest non-zero monomial of $g$ with respect to the lexicographic ordering by $x_1^{e_1} \cdots x_n^{e_n}$ (which we may assume to be monic). Assume that there is a monomial $m$ and $e \geq 0$ such that $(\kappa x_1^s)^e(m x_1^{e_1} \cdots x_n^{e_n} \delta) = \delta$ then it follows that any other monomial of $x_1^{s\frac{p^e -1}{p-1}} m g$ has exponent $< p^e -1$ in some $x_i$ and is thus annihilated by $\kappa^e$.

In particular, since we only need to consider monomials we may assume that $s = p-2$. We assert that there is $e \geq 0$ such that \[\frac{p-2}{p-1} (p^e - 1 +pe_1 - e_1) \leq p^e -1.\] Equivalently, we have to show that \[\frac{p-2}{p-1}  \leq \frac{p^e -1}{p^e - 1 +pe_1 - e_1}.\] But since the right hand side converges to $1$ for $e \to \infty$ this holds for all $e \gg 0$. Fix one $e'$ that satisfies this inequality. Then for $e$ such that $p^e -1 \geq e_i$ for all $i=2, \ldots, n$ and $e \geq e'$ we find a monomial $m$ as desired.

Finally, let us assume that $s = p-1$. Clearly, $x$ is a test element for $\omega_R$ so that, using \cite[Lemma 4.1]{blicklestaeblerbernsteinsatocartier}, we obtain $\tau(\omega_R, \kappa x^{p-1}) = (\kappa x^{p-1})^e x \omega_R = \kappa^e x^{p^e} \omega_R = x \omega_R$ for all $e \gg 0$.
\end{proof}

\begin{Bsp}
Let $k$ be an $F$-finite field and $R$ smooth over $k$. Fix a smooth hypersurface $x \in R$ and denote $R[t]/(t^{p-1} - x)$ by $S$. We write $f: \Spec S \to \Spec R$ for the morphism induced by the inclusion. Note that $f$ is a Kummer covering which is \'etale over $D(x)$. We consider the Cartier module $\omega_R$, where the Cartier structure is given by $\kappa x^s$ for some $1 \leq s \leq p -2$.

More precisely, the adjoint $C$ of $\kappa$ is induced from the fixed isomorphism $k \to F^! k$ by $s^!$, where $s: \Spec R \to \Spec k$ is the structural map. In particular, the adjoint of $\kappa x^s$ is given by $\omega_R \xrightarrow{x^s} \omega_R \xrightarrow{C} F^! \omega_R$. If we endow $\omega_R$ with its natural Cartier structure $\kappa$ then $f^!$ induces an isomorphism of Cartier modules $f^! \omega_R \cong \omega_S$, where $\omega_S$ also carries its natural Cartier structure $\kappa_S$. In particular, if we endow $\omega_R$ with the structure $\kappa x^s$ then this induces the Cartier structure $\kappa_S x^s = \kappa_S t^{s(p-1)}$ on $\omega_S$. From now on we will simply write $\kappa$ for the natural Cartier structure on a canonical module.

By virtue of Lemma \ref{PushforwardLocallyConstantSheafLNil} we conclude for $p \geq 3$ that $f^! \omega_R$ with Cartier structure induced by $\kappa x$ is $F$-pure but not $F$-regular, while $(\omega_R, \kappa x)$ is $F$-regular. By the same token if $s > 1$ then $(\omega_S, \kappa x^s)$ is not $F$-pure.
\end{Bsp}

We can also transfer this example to a context which fits in the framework where Stadnik constructs a notion of $V$-filtration (cf. \cite[Theorem 3.15, Definition 3.11]{stadnikvfiltrationfcrystal}).

\begin{Bsp}
We use the notations and assumptions from the previous example. First we show that taking the twisted inverse image of $(\omega_{R_x},  \kappa x^{s})$  along the restriction of $f$ yields an isomorphism with $(\omega_{S_x}, \kappa)$. We have $f^! \omega_R = \omega_S$ and a commutative diagram:
\[\begin{xy} 
\xymatrix{F_\ast \omega_{S_x}\ar[d]^{x^{\frac{s}{p-1}}} \ar[r]^{\kappa} & \omega_{S_x} \ar[d]^{x^{\frac{s}{p-1}}} \\ F_\ast \omega_{S_x} \ar[r]^{\kappa x^s} & \omega_{S_x}} \end{xy} \] Since $x^{\frac{s}{p-1}}$ is a unit the claim follows.

As $f\vert_{D(x)}$ is \'etale this shows that $(\omega_{R_x}, \kappa x^s)$ corresponds to a unit $R[F]$-crystal on $D(x)$ (equivalently to a locally constant sheaf on the \'etale site).

By Lemma \ref{PushforwardLocallyConstantSheafLNil} we have a local nil-isomorphism $j_\ast \omega_{R_x} \to \omega_R$. Hence, if we denote the open immersion $D(x) \to \Spec R$ by $j$ then $\tau(f^! j_\ast \omega_{R_x}) \neq f^!(\tau(j_\ast \omega_{R_x})$ by the previous example.
\end{Bsp}

\begin{Def}
Given a finite morphism $f: \Spec S \to \Spec R$ and an $R$-module $M$ we define the \emph{trace} as the $R$-linear map $\Tr: f_\ast f^! M \to M, \varphi \mapsto \varphi(1)$.
\end{Def}

\begin{Le}
\label{TraceFPureFaithfullyFlat}
Let $f: \Spec S \to \Spec R$ be a finite flat surjective morphism and let $M$ be a Cartier module via some Cartier algebra $\mathcal{C}$. Then $\Tr(f_\ast \underline{f^! M}) = \underline{M}$.
\end{Le}
\begin{proof}
By abuse of notation we denote the induced Cartier algebra on $f^!M$, which acts for a homogeneous element $\kappa$ of degree $e$ as $\kappa \cdot \varphi \mapsto \kappa \circ \varphi \circ F^e$, again by $\mathcal{C}$. It suffices to show that $\Tr(f_\ast \mathcal{C}_+ f^!M) = \mathcal{C}_+ M$. If $\kappa$ is homogeneous of degree $e$ then $\Tr(\kappa \varphi)) = \kappa(\varphi(F^e(1)) = \kappa(\varphi(1)) = \kappa(\Tr(\varphi))$ for $\varphi \in f^!M$. Since $f$ is flat and surjective $1$ is locally part of a free basis of $f_\ast S$ so that $\Tr$ is surjective.
\end{proof}

\begin{Le}
\label{TraceTestModuleFaithfullyFlat}
Let $f: \Spec S \to \Spec R$ be a finite flat surjective morphism and let $M$ be a Cartier module via some Cartier algebra $\mathcal{C}$. Assume that $f^! M$ and $M$ admit a common test element $x$. Then $\Tr(f_\ast \tau(f^! M)) = \tau(M)$.
\end{Le}
\begin{proof}
As in the last lemma we will denote the induced Cartier algebra on $f^!M$ again by $\mathcal{C}$. Then $\Tr(\tau(f^!M)) = \Tr(\sum_{e \geq 1} \mathcal{C}_e x \underline{f^!M})$. In particular, for a homogeneous element $\kappa$ of degree $e$ and $\varphi \in \underline{f^!M}$ we have $\Tr(\kappa x\varphi) = \kappa(x \varphi(F^e(1))) = \kappa x \Tr(\varphi)$. Using Lemma \ref{TraceFPureFaithfullyFlat} we obtain $\Tr(\underline{f^!M}) = \underline{M}$ so that $\Tr(\tau(f^!M)) = \sum_{e \geq 1} \mathcal{C}_e x \underline{M} = \tau(M)$.
\end{proof}

\begin{Le}
\label{FiniteFiberExtension}
Let $f: \Spec S \to \Spec R$ be a finite morphism of noetherian rings with $R$ reduced. If $Q = f(P)$ is generic for $P \in \Spec S$ generic then the induced morphism on stalks $\varphi: R_Q \to S_P$ is finite.
\end{Le}
\begin{proof}
Since $R$ is reduced $R_Q$ is a field and $Q_Q = 0$. We get a finite morphism of fibers $f^{-1}(Q) \to Q$ which is of the form $R_Q \to S[(R \setminus Q)^{-1}] =:  S_Q$. We have a factorization $R_Q \to S_Q \to S_P$, were $S_Q$ is an Artinian $R_Q$-algebra. In particular, it is a finite direct product of local Artinian algebras. One of these factors corresponds to $S_P$. We conclude that $S_P$ is finite over $R_Q$.
\end{proof}

\begin{Prop}
\label{ShriekFPureSupport}
Let $f: \Spec S \to \Spec R$ be a finite flat morphism and let $M$ be an $F$-pure coherent Cartier module. Then $\Supp \underline{f^! M} = \Supp f^!M$.
\end{Prop}
\begin{proof}
By \cite[Lemma 6.24]{staeblertestmodulnvilftrierung} we have $\Supp f^! M = f^{-1} \Supp M$. As $M$ is $F$-pure the annihilator $\Ann M$ is radical and we may base change and assume that $\Supp M = \Spec R$. Indeed, denote by $i$ the closed immersion $\Spec R/\Ann(M) \to \Spec R$ then applying $f'^!$ to the isomorphism $\id \to i^! i_\ast$ and using that $f'^! i^! \cong i'^! f^!$ we get, after adjunction, an induced morphism $i'_\ast f'^! \to f^! i_\ast$ which one verifies to be an isomorphism of $S$-modules. Then \cite[Proposition 3.3.23]{blickleboecklecartiercrystals} yields that this is an isomorphism of Cartier modules.

Let now $\eta$ be a generic point of a component of $\Supp f^!M = \Spec S$. Then $f(\eta)$ is a generic point of a component of $\Spec R$ by flatness. Consider the following commutative diagram \[\begin{xy} \xymatrix{ \Spec S_{\eta} \ar[r]^u \ar[d]_{f_\eta} & \Spec S \ar[d]_f \\ \Spec R_{f(\eta)} \ar[r]^v& \Spec R} \end{xy}\]
Note that $f_\eta$ is finite flat and surjective ($R_{f(\eta)}$ is a field since $R$ is reduced, finiteness follows from Lemma \ref{FiniteFiberExtension}). Since both $u, v$ are essentially \'etale we have $u^! f^! M = u^\ast f^! M = (f^! M)_\eta \cong f^!_\eta v^! M = f^!_\eta M_{f(\eta)}$.

Recall that the operation $\underline{\phantom{M}}$ commutes with localization.
Since $f_\eta$ is finite faithfully flat we obtain via Lemma \ref{TraceFPureFaithfullyFlat} a surjection $\Tr: {f_\eta}_\ast \underline{f_\eta^! M_{f(\eta)}} \to \underline{M}_{f(\eta)}= M_{f(\eta)}$. In particular, $\underline{(f^! M)_{\eta}} \neq 0$ thus $\eta \in \Supp \underline{f^!M}$. Finally, $\underline{f^! M}$ being coherent the support is closed and contains all generic points of the components of $\Spec S$ so that $\Supp \underline{f^! M} = \Spec S$. 
\end{proof}

\begin{Prop}
\label{FiniteFlatShriekEasyFregularDirection}
Let $f: \Spec S \to \Spec R$ be a finite flat and surjective morphism and let $M$ be an $F$-pure Cartier module. If $\underline{f^! M}$ is $F$-regular, then $M$ is also $F$-regular.
\end{Prop}
\begin{proof}
We use the characterization of $F$-regularity given by \cite[Proposition 5.2]{staeblertestmodulnvilftrierung}. First, we claim that given $c \in R$ such that $D(c) \cap \Supp M$ is dense in $\Supp M$ then $f^{-1}(D(c)) \cap \Supp f^! M$ is dense in $\Supp \underline{f^! M} = \Supp f^!M$, where the equality is due to Proposition \ref{ShriekFPureSupport}. By \cite[Lemma 6.24]{staeblertestmodulnvilftrierung} we have $f^{-1}(\Supp M) = \Supp f^!M$ so that the claim is that $f^{-1}(\Supp M \cap D(c))$ is dense in $f^{-1}(\Supp M)$. But by flatness (make a base change) and the density of $\Supp M \cap D(c)$ in $\Supp M$ all generic points of $f^{-1}(\Supp M)$ are contained in $f^{-1}(\Supp M \cap D(c))$. 

Let now $c \in R$ be a test element for $M$. Then by the above and our assumptions $c$ is also a test element for $\underline{f^! M}$. Using Lemmata \ref{TraceFPureFaithfullyFlat} and \ref{TraceTestModuleFaithfullyFlat} we have $M = \underline{M} = \Tr(f_\ast(\underline{f^!M})) = \Tr(f_\ast f^! \tau(M)) = \tau(M)$.
\end{proof}

\section{A comparison with Stadnik's $V$-filtration}
\label{VFiltComparison}

The goal of this section is to relate the filtration of Definition \ref{LimitTestmodulefiltration} above with the (super-specializing) $V$-filtration for unit $F$-modules introduced by Stadnik in \cite[Definitions 3.4, 3.6]{stadnikvfiltrationfcrystal} in the cases where Stadnik proved existence of his filtration. We will also see that they do not coincide in general. Let us begin by recalling Stadnik's definition:

\begin{Def}
\label{StadnikVFiltration}
Let $\mathcal{M}$ be a unit $F$-module on $\Spec R$ with adjoint structural map $\alpha: M \to F_\ast M$. A \emph{super-specializing $V$-filtration} along $I=(f) \subseteq R$ is an exhaustive separated discrete left-continuous $\mathbb{Q}$-indexed filtration $V^{\cdot} \mathcal{M}$ of $R$-modules such that
\begin{enumerate}[(i)]
 \item $V^0 \mathcal{M}$ is coherent.
\item $I V^i \mathcal{M} = V^{i+1} \mathcal{M}$ for all $i \in \mathbb{Q}, i \neq -1$.
\item $\alpha(V^i) \subseteq F_\ast V^{ip}$ for all $i \in \mathbb{Q}$.
\item $F_{R/I}$ induces an isomorphism $F_{R/I}^\ast Gr^i \mathcal{M} \cong Gr^{ip} \mathcal{M}$ whenever $Gr^i \mathcal{M} \neq 0$.
\item Multiplication by $f$ induces an isomorphism $\mu_f: Gr^i \mathcal{M} \to Gr^{i+1} \mathcal{M}$ for all $i \in \mathbb{Q}, i \neq -1$.
\end{enumerate}
\end{Def}

For a normal noetherian integral domain $R$ and a morphism $f: \Spec S \to \Spec R$ we say that $f$ is \emph{tamely ramified} along $x \in R$ if $f$ is finite and \'etale over $D(x)$, S is normal, $x$ is smooth and every irreducible component of $\Spec S$ dominates $\Spec R$. Moreover, we require that for every generic point $p \in V(x)$ the extension $R_{p} \subseteq S_{f^{-1}(p)}$ is tamely ramified (see \cite[Definition 2.2.2]{grothendieckmurretamefun}). Important for computations is the fact that \'etale locally a tamely ramified covering is a disjoint union of Kummer coverings. More precisely, for every $p \in \Spec R$ there is an \'etale neighborhood $U \to \Spec R$ such that $\Spec S \times_{\Spec R} U \to U$ is a disjoint union of Kummer coverings (\cite[Corollary 2.3.4]{grothendieckmurretamefun}).

\begin{Bem}
\label{StadnikVfiltrationRemark}
Let us briefly elaborate on axiom (iv): Note that axiom (ii) implies that $Gr^i \mathcal{M}$ is an $R/I$-module. Next note that $\alpha: V^i \mathcal{M} \to F_\ast V^{ip} \mathcal{M}$ induces a morphism $Gr^i \mathcal{M} \to F_\ast Gr^{ip} \mathcal{M}$. In particular, we have a morphism $Gr^i \mathcal{M} \to {F_{R/I}}_\ast Gr^{ip} \mathcal{M}$ whose adjoint $F_{R/I}^\ast Gr^i \mathcal{M} \to Gr^{ip} \mathcal{M}$ is the desired isomorphism.

Also note that axiom (v) is already implied by (ii) provided that $f$ is a non-zero divisor on $\mathcal{M}$. Indeed, in this case (ii) yields isomorphisms $\mu_f: V^i \mathcal{M} \to V^{i+1} \mathcal{M}$ and similarly an isomorphism $V^{i + \eps} \mathcal{M} \to V^{i +1 + \eps} \mathcal{M}$. Thus it induces an isomorphism on the associated graded.

Stadnik proves that such a filtration is unique (see \cite[Proposition 3.8]{stadnikvfiltrationfcrystal}) and furthermore shows existence in the following special case: $\mathcal{M}$ is of the form $j_\ast \mathcal{N}$ for some coherent unit $F$-module $\mathcal{N}$ on $D(f)$, $f$ smooth and for every closed point $z \in V(f)$ there is a (Zariski) neighborhood $W$ of $z$ and a tamely ramified covering $Y \to W$ which trivializes $\mathcal{M}\vert_{D(f) \cap W}$ (see \cite[Theorem 3.15]{stadnikvfiltrationfcrystal}).
\end{Bem}

Following Stadnik we turn this last notion into a definition:

\begin{Def}
\label{TameUnitFCrystal}
If $R$ is smooth over a field $k$, $f \in R$ a smooth hypersurface, $j: D(f) \to \Spec R$ and $\mathcal{N}$ is a unit $R_f[F]$-crystal then we call $j_\ast \mathcal{N}$ a \emph{tame unit $R[F]$-crystal} if for every closed point $z \in V(f)$ there is a (Zariski) neighborhood $W$ of $z$ and a tamely ramified covering $Y \to W$ which trivializes $\mathcal{M}\vert_{D(f) \cap W}$.
\end{Def}

\begin{Le}
\label{ClosedImmersionCartierStructureGraded}
Let $R$ be smooth over an $F$-finite field. If $x \in R$ is smooth then $(\omega_R/x\omega_R, \kappa x^{p-1})$ is isomorphic to $(\omega_{R/(x)}, \kappa)$.
\end{Le}
\begin{proof}
Up to a shift one has $i^! \omega_R \cong \omega_R/x \omega_R$ with Cartier structure $\kappa x^{p-1}$, where $i: \Spec R/x \to \Spec R$ is the closed immersion (see \cite[Example 3.3.12]{blickleboecklecartiercrystals}). Then if $s: \Spec R \to \Spec k$ and $s': \Spec  R/(x) \to \Spec k$ are the structural morphisms one has $i^! s^! = s'^!$ and both $\kappa$ on $\omega_{R/(x)}$ and on $\omega_R$ are induced by shrieking the isomorphism $k \to F^! k$.
\end{proof}

\begin{Le}
\label{AssocGradedLimitEqualsAssocGraded}
Assume that $R$ is regular, essentially of finite type over an $F$-finite field, $M$ a Cartier module satisfying $M = \overline{M}$ and $f \in R$ a non-zero divisor. Then \[\tau(\mathcal{M}, f^{t-\eps})/\tau(\mathcal{M}, f^t) = \tau(M, f^{t - \eps})/\tau(M, f^t),\] where $\mathcal{M} = \colim_e {F^e}^! M$.
\end{Le}
\begin{proof}
Since $M = \overline{M}$ the natural map $\gamma: M \to \mathcal{M}$ is injective. Now we have, by definition, $\gamma(\tau(M, f^t)) = \tau(\mathcal{M}, f^t)$ which yields the claim.
\end{proof}

Also recall (\cite[Section 4]{staeblertestmodulnvilftrierung}) that for a Cartier module $(M, \kappa)$ the quotient $Gr^t_\tau = \tau(M, f^{t - \eps})/\tau(M, f^t)$ naturally carries a Cartier structure induced by $\kappa f^{\lceil t (p-1)\rceil}$. By Lemma \ref{CoherentModelOfPushforwardExists} the assumption $M = \overline{M}$ is satisfied in the situation where $(N, \kappa)$ is a Cartier module on $R_f$ such that $C:N \to F^!N$ is an isomorphism an $M \subseteq j_\ast N$ is a local nil-isomorphism.

\begin{Prop}
\label{TestModuleVfiltrationShiftJustification}
Let $R$ be smooth of finite type over a perfect field, let $x \in R$ be a smooth hypersurface and let $j: D(x) \to \Spec R$ be the associated open immersion. Then for $t \in \mathbb{Q}$ there is $0 < \eps \ll 1$ such that the filtration $\tau(j_\ast \omega_{R_x}, x^{t + 1 -\eps})$ (as in Definition \ref{LimitTestmodulefiltration}) coincides with $V^t(j_\ast R_x) \otimes \omega_R$ filtered along $x$ (Defintion \ref{StadnikVFiltration}).

Moreover, under the equivalence of unit $R/(x)[F]$-modules with Cartier crystals $Gr_V^0$ and $Gr_\tau^{1}$ correspond to each other.
\end{Prop}
\begin{proof}
Recall that the unit $F$-module corresponding to $j_\ast \omega_{R_x}$ is just $j_\ast R_x$ endowed with the (adjoint of the) ordinary Frobenius. In this case it is easy to see that the filtration $V^t := \{ f \, \vert \, \nu(f) \geq t\}$, where $\nu: Q(R) \to \mathbb{Z}$ is the valuation attached to $R_{(x)}$, is a super-specializing $V$-filtration. In other words, one has $V^t = R \cdot x^{\lceil t \rceil}$ for all $t \in \mathbb{Q}$.

On the other hand, $j_\ast \omega_{R_x}$ is locally nil-isomorphic to $\omega_{R} \otimes R \frac{1}{x}$ and the test module filtration is given by $\tau(j_\ast \omega_{R_x}, x^t) = x^{\lfloor t \rfloor} \omega_R$ for $t \geq 0$. For $t < 0$ we set $\tau(j_\ast \omega_{R_x}, x^t) = x^{\lfloor t \rfloor} \omega_R$ according to \ref{DefTestModuleFiltrationQuasicoherentPushforward}.
So if $t \geq 0$ is not integral then we take $\eps < \{t \}$ and obtain $x^{\lfloor t + 1 - \eps \rfloor} = x^{\lceil t \rceil}$ as desired. If $t \geq 0$ is integral then we take any $0 < \eps < 1$. 

One easily verifies that the induced unit $R/(x)[F]$-structure on $Gr_V^0 = R/(x)$ is in this case given by the ordinary Frobenius $R/(x) \to F_\ast R/(x)$. Tensoring with $\omega_{R/(x)}$ then induces the ordinary Cartier structure on $\omega_{R/(x)}$ which by Lemma \ref{ClosedImmersionCartierStructureGraded} coincides with the one on $Gr^{1}$ (cf.\ \cite[Proposition 4.5]{staeblertestmodulnvilftrierung}).
\end{proof}

Based on this result one might hope that given $t$ and a unit $R[F]$-module $\mathcal{M}$ there is $0 < \eps \ll 1$ such that $\tau(\mathcal{M} \otimes \omega_R, x^{t+1 - \eps}) = V^t(\mathcal{M}) \otimes \omega_R$, where the left-hand side is the filtration introduced in \ref{LimitTestmodulefiltration} and the right-hand side is Stadnik's filtration. However, this is not the case in general.

First we give an example that shows that Stadnik's $V$-filtration does not commute with graph embeddings (the corresponding result \emph{does} hold for the test module filtration -- see \cite[Proposition 3.7]{staeblertestmodulnvilftrierung}).

\begin{Bsp}
Let $k$ be a perfect field and consider the embedding $i: \Spec k \to \Spec k[t], t \mapsto 0$. Consider $k$ as a unit $F$-module via $F^\ast k \to k$. Observe that the trivial filtration (i.e.\ $V^t = k$ for all $t$) is a $V$-filtration along $f = 1$. 

Next, $i_+ k = k[t, t^{-1}]/k[t]$ is a unit $F$-module via the Frobenius induced from localization. We claim that $i_+ k$ does not admit a $V$-filtration along $t+1$ in the sense of Definition \ref{StadnikVFiltration} (note that this is up to an identification just the graph embedding of $k$ filtered along $1$). Indeed, one has $\Supp i_+(k) = V(t)$ but assuming that $i_+k$ admits a $V$-filtration along $t+1$ the associated graded of this filtration (if it is non-zero) is supported on $V(t+1)$. We conclude that $V^s$ must be the trivial filtration. One can now check that multiplication by $t+1$ is an automorphism (the element $\frac{a}{t^n} - \frac{a}{t^{n-1}} + \frac{a}{t^{n-2}} + \cdots + (-1)^{n+1} \frac{a}{t}$ is the preimage of $\frac{a}{t^n}$). Hence, this filtration satisfies all of Stadnik's axioms except for (i) -- i.e.\ $V^0$ is not finitely generated as a $k[t]$-module.
\end{Bsp}

\begin{Bsp}
\label{StadnikVFiltrationNotEqualTestModuleFiltration}
Let $k$ be an $F$-finite field, $R=k[x]$ and consider the trivial Cartier module $(\omega_R, \kappa)$ and take a graph embedding $i: \Spec R \to \mathbb{A}^1_R = \Spec k[x,t]$ sending $t$ to $x$. Then $\tau(i_\ast \omega_R, t^s) = i_\ast \tau(\omega_R, x^s) = x^{\lfloor s \rfloor} i_\ast \omega_R$ (see \cite[Proposition 3.7]{staeblertestmodulnvilftrierung}).

Note that the unit $F$-module corresponding to $i_\ast \omega_R$ is given by $i_+ R$ which admits the explicit description $i_+ R = R[t]_{x-t}/R[t]$ with adjoint structural map \[\alpha: i_+R \longrightarrow F_\ast i_+ R,\quad x \longmapsto x^p,\]
i.e.\ it is just the map induced by the Frobenius on $R[t]_{x -t}$. The embedding $i_\ast R \to i_+ R$ is given by $r \mapsto \frac{r}{x -t}$.  Tensoring with $\omega_R$ the map $i_\ast R \to i_+ R, r \mapsto \frac{r}{f-t}$ corresponds to the natural map $i_\ast \omega_R \to \colim {F^e}^! i_\ast \omega_R$.

In particular, we obtain $\tau(\colim {F^e}^! i_\ast \omega_R, x^s) \otimes \omega_{R[t]}^{-1} = R[t] \frac{x^s}{x - t}$. We want to argue that $\tau(\colim {F^e}^! i_\ast \omega_R, x^{s+1 - \eps}) \otimes \omega_{R[t]}^{-1} \neq V^s$. 

First, we will show that (iii) of Definition \ref{StadnikVFiltration} is not satisfied. Fix $s \in\mathbb{N}$ and choose $0 < \eps \ll 1$. We want to show that $\alpha(\tau(\colim {F^e}^! i_\ast \omega_R,x^{s+1 - \eps})) \subseteq \tau(x^{sp +1 - \eps})$ does not hold in general. Assume to the contrary that this inclusion holds. Then there exists $f \in R[x,t]$ depending on $s$ such that $\frac{x^{sp}}{x^p - t^p} = f \frac{x^{sp}}{x-t}$ in $i_+ R$. Equivalently, there exists an $f$ such that \[\frac{x^{sp} - x^{sp}(x-t)^{p-1}f}{x^p - t^p} \in R[t].\] In other words, we must have $x^{sp} \in (x^p - t^p, x^{sp}(x-t)^{p-1})$ in $R[t]$. Applying the automorphism $t \mapsto t + x$ of $R[t]$ we see that equivalently \[x^{sp} \in (t^p, x^{sp}t^{p-1})\] which is never satisfied\footnote{In fact, since $x$ is smooth we could have worked with the embedding $t \mapsto 0$ from the beginning.}.

Next we claim that, in the setting of Cartier modules, the natural map \[Gr^1 i_\ast \omega_R \to F^!_{k[x]} Gr^1 i_\ast \omega_R = \Hom_R(F_\ast R, Gr^1 i_\ast \omega_R)\] is not surjective and in particular not an isomorphism. It then follows that the corresponding unit $F$-module is not coherent which shows that (iv) in Definition \ref{StadnikVFiltration} is not satisfied for $i =0$.
By the above computation $Gr^1 i_\ast \omega_R = \omega_R/x \omega_R = k$ as an $R$-Cartier module with Cartier structure induced by $\kappa x^{p-1}$.

The natural map is then given by \[r \mapsto [f \mapsto \kappa(x^{p-1} r f dx) \text{ mod } x\omega_R]\]
If this were an isomorphism then $Gr^1 i_\ast \omega_R$ would correspond to a locally constant sheaf. That is, we would have a finite \'etale morphism $\varphi: \Spec S \to \Spec R[x]$ such that $\varphi^!Gr^1 i_\ast \omega_R \cong \omega_S$ as Cartier modules. But since $Gr^1 i_\ast \omega_R$ is not even invertible this is absurd.

In more down to earth terms this can be seen as follows. Assume that $k$ is perfect. Then the map evaluates for $r \in k$ to $r^{\frac{1}{p}} \kappa(x^{p-1} f dx)$ mod $x \omega_R$. In particular, for $f =x$ we obtain the zero map. But $F_\ast R$ is free with basis $1, x, \ldots, x^{p-1}$ so that the projection $x \mapsto 1$ is not contained in the image.
\end{Bsp}

What we will achieve in the following is that the two filtrations coincide (up to left/right-continuity) in the setting where Stadnik proved existence. Stadnik constructs his $V$-filtration by taking $G$-invariants of the trivial filtration in the situation of a Kummer covering and then shows that his $V$-filtration only has to be constructed \'etale locally. In order to achieve a comparison we will therefore show that test module filtrations can similarly be obtained by taking $G$-invariants.

\begin{Le}
\label{InvariantsContainedInTestModule}
Let $R$ be $F$-finite and $M$ a Cartier module via some Cartier algebra $\mathcal{C}$. Let $f: \Spec S \to \Spec R$ be a finite flat morphism such that $S$ is free as an $R$-module with basis $\{1, s_1, \ldots, s_n\}$. Assume furthermore that the $s_i^p$ are contained in the $R$-module generated by the $s_1, \ldots, s_n$. If $M$ and $f^!M$ admit a common test element $x$ then the $R$-module $N = \{ \varphi \in f^!M \, \vert \, \varphi(1) \in \tau(M), \varphi(s_i) = 0\}$ is contained in $\tau(f^!M)$.
\end{Le}
\begin{proof}
First, we show that $N' = \{\varphi \in f^!M \, \vert \, \varphi(1) \in \underline{M}, \varphi(s_i) = 0\}$ is contained in $\underline{f^!M}$. In order to do this it suffices to show that $N' \subseteq \mathcal{C}_+ N'$. So fix $\varphi \in N'$. As $\underline{M}$ is $F$-pure we find finitely many homogeneous elements $\kappa_e \in \mathcal{C}_e$ ($e \geq 1$) and $m_e \in \underline{M}$ such that $\sum_e \kappa_e(m_e) = \varphi(1)$. We define homomorphisms $\varphi_{m_e}: S \to M$ by $\varphi_{m_e}(1) = m_e$ and $\varphi_{m_e}(s_i) = 0$. Then $\sum_e \kappa_e \varphi_{m_e}(1) = \sum_e \kappa_e(\varphi_{m_e}(F^e(1))) = \varphi(1)$ and $(\sum_e \kappa_e \varphi_{m_e})(s_i) = \sum_e \kappa_e(\varphi_{m_e}(s_i^{p^e})) = 0$ since $s_i^p = \sum_i r_i s_i$ by assumption. We conclude that $\varphi$ lies in $\mathcal{C}_+ N'$.

Next we show that given $\varphi \in N$ we can find $\psi_e \in \underline{f^!M}$ such that $\varphi = \sum_{e =1}^n \kappa_e x \psi_e$ for some $\kappa_e \in \mathcal{C}_e$. Since $x$ is also a test element for $M$ we can write $\tau(M) = \sum_{e \geq 1} \mathcal{C}_e x \underline{M}$. In particular, we find $\kappa_e \in \mathcal{C}_e$ and $m_e \in \underline{M}$ such that $\varphi(1) = \sum_{e =1}^n \kappa_e (x m_e)$. Define homomorphisms $\psi_e: S \to M$ via $\psi_{e}(1) = m_e$ and $\psi_{e}(s_i) = 0$. By the above $\psi_e \in \underline{f^!M}$ and one verifies that the $\psi_e$ satisfy the claimed equation.
\end{proof}

\begin{Bem}
The assumption on the basis in the lemma is in particular satisfied if $f: \Spec S \to \Spec R$ is a Kummer covering. Indeed, say $S = R[t]/(t^n -a)$ then $1, t^1, \ldots, t^{n-1}$ is an $R$-basis of $S$ and for $i=1. \ldots, n-1$ we can write $ip = nl +r$ and $t^{ip} = a^l t^r$. Since $(n,p) =1$ we must have that $r \neq 0$.
\end{Bem}

\begin{Prop}
\label{KummerTestModuleGinvariants}
Let $R$ be an $F$-finite ring containing an algebraically closed field and $x$ a non zero-divisor. Let $(M, \kappa)$ be an $F$-pure Cartier module, assume that $M_x$ is $F$-regular and that $x$ is a non zero-divisor on $M$. If $f: \Spec S \to \Spec R$ is a Kummer covering along $V(x)$ with Galois group $G$ then one has an isomorphism of $\mathcal{C} = \langle \kappa^e x^{\lceil tp^e\rceil} \, \vert \, e \geq 0 \rangle$-modules $\tau(f^!M, x^t)^G \cong \tau(M, x^t)$ for any $t \in \mathbb{Q}_{\geq 0}$
\end{Prop}
\begin{proof}
By Proposition \ref{ShriekFPureSupport} we have $\Supp f^!M = \Supp \underline{f^! M}$. As $x$ is a non zero-divisor of $f^!M$ we obtain that $D(x) \cap \Supp \underline{f^!M} \subseteq \Supp \underline{f^!M}$ is dense (see  the proof of \cite[Proposition 4.2]{staeblertestmodulnvilftrierung} for the argument). Since $f$ is \'etale over $D(x)$ and $M_x$ is $F$-regular \cite[Theorem 6.15]{staeblertestmodulnvilftrierung} implies, together with the above, that $x$ is test element for $f^! M$.

As $f$ is a Kummer covering $S$ is of the form $R[t]/(t^n - x)$ for some $n$ with $(n,p) =1$.
Next, note that $G$ is cyclic of order $n$. The natural $G$-action on $S$ is given by letting a generator $\sigma$ act by multiplication with a primitive $n$th root of unity $\zeta_n$ on $t$. If $\varphi \in f^! M$ is such that $\varphi(t^i) \neq 0$ for some $1 \leq i \leq {n-1}$, then $\sigma \varphi(t^i) = \zeta_n^{n-i} \varphi(t^i)$ so that $\varphi$ is not invariant (note that we consider the corresponding right action on $\Hom_R(S, -)$). But clearly, $\sigma \varphi(1) = \varphi(1)$ so that the invariants are those morphisms than send $t^i$ to zero for $1 \leq i \leq n-1$. 

Using Lemmata \ref{InvariantsContainedInTestModule} and \ref{TraceTestModuleFaithfullyFlat}, we obtain that the module $N$ constructed in \ref{InvariantsContainedInTestModule} coincides with $\tau(f^! M)^G$. Hence, evaluation at $1$ induces the desired isomorphism. Moreover, this isomorphism is compatible with Cartier structures. Indeed, given $\varphi \in \tau(f^!M)^G$ we have $\kappa^e x^{\lceil tp^e\rceil}(\varphi(F^e(1))) = \kappa^e x^{\lceil tp^e \rceil}(\varphi(1))$.
\end{proof}

\begin{Le}
Let $f: \Spec S \to \Spec R$ be a finite \'etale morphism. Then $S \to \Hom_R(S,R), s \mapsto \Tr \circ \mu_s$, where $\mu_s: S \to S$ is multiplication by $s$, is an isomorphism of $S$-modules. 
\end{Le}
\begin{proof}
See \cite[1.4, 1.2 and Proposition 6.9]{lenstragaloistheory}.
\end{proof}

\begin{Le}
\label{EtalePullbackShriekGEquivariant}
 Let $f: \Spec S \to \Spec R$ be a finite \'etale morphism of integral normal schemes. Let $M$ be an $R$-module and $G$ the Galois group of the fraction field extension $Q(S)/Q(R)$. If $G$ acts on $f^\ast M$ by its natural action on $S$ and $G$ acts on $\Hom_R(S, M)$ via $\sigma \varphi \mapsto \varphi \circ \sigma^{-1}$ then $f^\ast M \to \Hom_R(S,M), m \otimes s \mapsto [t \mapsto \Tr(st) m] $ is $G$-equivariant.
\end{Le}
\begin{proof}
Since $R$ and $S$ are normal domains we have $\Tr(x) = \sum_{\sigma \in G} \sigma(x)$. Hence, $\Tr(\sigma(s) t) = \Tr(s \sigma^{-1} (t))$.
\end{proof}

Let us recall that if $R$ is essentially of finite type over an $F$-finite field $k$ with structural map $s$ then $s^! k = \omega_R$ is an invertible sheaf and by our conventions is endowed with an isomorphism $\omega_R \to F^! \omega_R$ (cf.\ Introduction).

\begin{Bsp}
\label{ExampleKummer}
Write $f: \Spec S \to \Spec R$ for the morphism associated to the Kummer covering $R \subseteq S = R[t]/(t^n -x)$ (in particular, $(n, p) = 1$). Assume furthermore that $R$ contains all $n$th roots of unity. We will denote by $\delta_i \in \Hom_R(S,R)$ for $0 \leq i \leq n-1$ the map which sends $t^i$ to $1$ and the other $t^j$ to $0$ for $0 \leq j \leq n-1$, $j \neq i$.

Then $\Hom_R(S, R) \to S, \delta_{n-1} \mapsto 1$ induces an isomorphism of $S$-modules since $t^i \delta_{n-1} = \delta_{n-1-i}$ for $0 \leq i \leq n-1$. Moreover, one has $\Tr = n \delta_0 = n t^{n-1} \delta_{n-1}$. The last equality is due to the fact that $\Tr(t^i) = \sum_{j = 1}^n \zeta^{ij} t^i =0$ for $i \neq 0$, where $\zeta$ is a primitive $n$th root of unity. 

Let us now additionally assume that $R$ is regular essentially of finite type over an $F$-finite field. If $j: \Spec R_x \to \Spec R$ denotes the open immersion and $M$ is an $R_x$-module then one has a chain of natural isomorphisms \[f^! (j_\ast M \otimes \omega_R) \cong j'_\ast f'^! (M \otimes \omega_{R_x}) \cong j'_\ast f'^\ast (M \otimes \omega_{R_x}) \cong f^\ast (j_\ast M \otimes \omega_R)\] Let us denote the map $S \to j_\ast M \otimes \omega_R$ sending $t^i$ to $m \otimes \omega$ and $t^j$ to $0$ for $j=0, \ldots, n-1, j\neq i$ by $\delta_{i, m, \omega}$. With this notation the isomorphism $\Hom_R(S, j_\ast M \otimes \omega_R) \to j'_\ast f'^\ast (M \otimes \omega_{R_x})$ is given by sending $\delta_{i, m, \omega}$ to $\frac{1}{n} t^{-i} \otimes m \otimes \omega$. This follows since $n t^{n-1} \delta_{n-1} = \Tr \mapsto 1$ and using $S$-linearity. Since $\frac{1}{n} t^{-i} \otimes m \otimes \omega = \frac{1}{n} t^{n-i} \otimes \frac{m}{x} \otimes \omega$ we see that the composition of the natural isomorphisms above is given by \begin{align}f^!(j_\ast M \otimes \omega_R) \longrightarrow f^\ast(j_\ast M \otimes \omega_R),\,\, & \delta_{i, m , \omega} \longmapsto \frac{1}{n} t^{n-i} \otimes \frac{m}{x} \otimes \omega.\end{align}

Moreover, if $R$ is normal and we endow $f^! (j_\ast M \otimes \omega_R) = \Hom(S, j_\ast M \otimes \omega_R)$ with its natural right action of the Galois group $G$ of the fraction field extension $Q(S)/Q(R)$ and $f^\ast( j_\ast M \otimes \omega_R)$ with its natural (left) $G$-action on $S$ then this isomorphism is $G$-equivariant.
\end{Bsp}

In what follows if $M$ is an $R[G]$-module and $N \subseteq M$ an $R$-submodule then we will denote $M^G \cap N$ by $N^G$ although $N$ may not be an $R[G]$-module.
\begin{Le}
\label{GInvariantsKummerTensor}
Let $f: \Spec S \to \Spec R$ be a Kummer covering of degree $n$ of integral schemes ramified over $V(x)$ and assume that $R$ is regular essentially of finite type over an $F$-finite field which contains all $n$th roots of unity. Let $j: \Spec R_x \to \Spec R$ be the open immersion induced by localization. Let $M$ be an $R_x$-module and $G$ the Galois group of the fraction field extension $Q(S)/Q(R)$. Then there is a natural isomorphism \[\Sigma: (f^\ast j_\ast  M)^G \otimes \omega_R \longrightarrow (f^!(j_\ast M \otimes \omega_R))^G.\]
If $N \subseteq f^\ast j_\ast M$ is an $S$-submodule then $\Sigma(N^G \otimes \omega_R) \cong (x^{\frac{n-1}{n}}N \otimes \omega_S)^G$.
\end{Le}
\begin{proof}
The natural isomorphism $\Sigma$ is given by restricting the inverse of natural isomorphism $(1)$ in Example \ref{ExampleKummer} to $G$-invariants.

For the second claim note that, upon identifying $\omega_S$ with $f^! \omega_R$ one has a natural isomorphism \[\Xi: f^!(j_\ast M \otimes \omega_R) \to f^\ast j_\ast M \otimes_S f^! \omega_R\] whose inverse is given by \[m \otimes s \otimes \varphi \mapsto  [t \mapsto \varphi(st) \otimes m].\] In particular, we obtain a $G$-action on $f^\ast j_\ast M \otimes_S f^! \omega_R$ by transport of structure.

Let us write $S = R[t]/(t^n -x)$. In particular, $S$ admits a free basis $1, t, \ldots, t^{n-1}$ and the inclusion $R \to S$ splits. We localize and assume that $\omega_R = R \cdot \omega$ is free of rank $1$. We will denote by $\delta_i$ the element of $\Hom_R(S, \omega_R)$ which sends $t^i$ to $\omega$ and $t^j$ to $0$ for $i \neq j, 0 \leq  j \leq n-1$. As in Example \ref{ExampleKummer} the element $\delta_{n-1}$ generates $\Hom_R(S, \omega_R)$ as an $S$-module.

We now come to the claimed isomorphism. An element of $N^G$ is also contained in $(f^\ast j_\ast M)^G$. Since $S$ is a free $R$-module with basis $t^i, 0\leq i \leq n-1$ an easy computation shows that $G$-invariant elements are of the form $m \otimes 1$ for $m \in j_\ast M$. In particular, an element of $N^G \otimes \omega_R$ is of the form $m \otimes 1 \otimes \omega$. Such an element is mapped via $\Xi \circ \Sigma$ to the $G$-invariant element $n t^{n-1}m \otimes \delta_{n-1} \in (t^{n-1} N \otimes \omega_S)^G$.

For the other direction assume that $g \in N$ and that $g \otimes s t^{n-1} \delta_{n-1} \in N \otimes \omega_S$ is $G$-invariant. Then $g \otimes s t^{n-1} \delta_{n-1} = s g \otimes \delta_0$. The $G$-invariants of $f^!(j_\ast M \otimes \omega_R)$ are of the form $[a \mapsto \delta_0(a) \otimes m]$ for some $m \in j_\ast M$. Under the natural isomorphism $f^! (j_\ast M \otimes \omega_R) \to f^\ast j_\ast M \otimes \omega_S$ such an element is mapped to $m \otimes 1 \otimes \delta_0$. We conclude that we have an equality (considered as elements of $f^\ast j_\ast M \otimes_S \omega_S$) $sg \otimes t^{n-1} \delta_{n-1} = m \otimes 1 \otimes t^{n-1} \delta_{n-1}$ for some $m \in j_\ast M$. From the isomorphism $S \otimes_S \omega_S \to S, s \otimes \delta^{n-1} \mapsto s$ and the fact that $t^{n-1}$ is a non zero-divisor we thus obtain $m \otimes 1= sg$. Hence, $g \otimes s t^{n-1} \delta_{n-1} = sg \otimes \delta_0 \in N^G \otimes \omega_R$.
\end{proof}

For the next lemma note that if $S = R[t]/(t^n - x)$ then $S/(t)$ and $R/(x)$ are isomorphic as rings.

\begin{Le}
\label{KummerGradedTestmoduleInclusion}
Assume that $R$ is essentially of finite type over an $F$-finite field. Let $f: \Spec S \to \Spec R$ be a Kummer covering ramified along $x \in R$. If $(M, \kappa)$ is an $F$-pure Cartier module on $R$ and $M$ and $f^!M$ admit a common test element then we have an inclusion of Cartier $R/(x)$-modules \[\tau(M, x^{1- \eps})/\tau(M, x^1) \longrightarrow \tau(f^!M, t^{1 - \eps})/\tau(f^! M, t^1),\] where the Cartier structure on the quotient is given by $\kappa x^{p-1}$ and $\kappa' t^{p-1}$, where $\kappa'$ is the induced Cartier structure on $f^!M$.
\end{Le}
\begin{proof}
Arguing similarly to \cite[Proposition 4.8]{staeblertestmodulnvilftrierung} one has an isomorphism \[\tau(f^!M, t^{1 - \eps})/\tau(f^! M, t^1) \cong \tau(f^!M, t^{n - \eps})/\tau(f^! M, t^n)\] of Cartier modules. Endowing $f_\ast \tau(f^!M, x)$ with Cartier structure $\kappa' x^{p-1}$ one checks that $\tau(f^!M, x)^G$ is a Cartier submodule. By Proposition \ref{KummerTestModuleGinvariants} we have an isomorphism $\tau(M, x) \cong \tau(f^!M, x)^G \subset f_\ast \tau(f^!M, x)$ and if we endow $\tau(M, x)$ with Cartier structure $\kappa x^{p-1}$ then this is an isomorphism of Cartier modules. A similar statement holds for $\tau(M, x^{1- \eps})$ and $\tau(f^!M, x^{1-\eps})$. Forming quotients one obtains an inclusion \begin{equation}\label{asdf}\tau(M, x^{1- \eps})/\tau(M, x^1) \to f_\ast(\tau(f^!M, t^{1 - \eps})/\tau(f^! M, t^1))  \end{equation} of Cartier modules since $\tau(f^!M, x) \cap \tau(f^!M, x^{1-\eps})^G = \tau(f^!M,x) \cap (f^!M)^G = \tau(f^!M, x)^G$.

We now have a commutative diagram \[\begin{xy} \xymatrix{\Spec S/(t) \ar[r]^{i'} \ar[d]^g & \Spec S \ar[d]^f\\ \Spec R/(x) \ar[r]^i & \Spec R} \end{xy}\] where $i, i'$ are closed immersions and where $g$ is the natural isomorphism.
Since $t \tau(f^!M, x^{1 - \eps}) \subseteq \tau(f^! M, x^1)$ for $\eps < \frac{1}{n}$ by Brian\c{c}on-Skoda (\cite[Theorem 4.21]{blicklep-etestideale}) and Proposition \ref{ExponentAndRootTestModule} we obtain that $\tau(f^! M, x^{1- \eps})/\tau(f^! M, x^1)$ is in fact an $S/(t)$-module. Similarly one argues that $\tau(M, x^{1 -\eps})/\tau(M, x^1)$ is an $R/(x)$-module. We can thus write the former as $i'_\ast N$ and the latter as $i_\ast \tilde{N}$. The inclusion \eqref{asdf} can then be denoted by $i_\ast \tilde{N} \to f_\ast i'_\ast N$. An application of $i^!$ to the isomorphism $f_\ast i'_\ast N \cong i_\ast g_\ast N$ and the fact that the counit $i^! i_\ast \to \id$ is an isomorphism of Cartier modules yield the inclusion $\tilde{N} \to g_\ast N$.
\end{proof}

\begin{Le}
\label{KummerGradedVFiltrationInclusion}
Let $R$ be smooth over an algebraically closed field and $x \in R$ smooth. Let $\mathcal{M}$ be a unit $F$-crystal on $D(x)$ and denote by $j: D(x) \to \Spec R$ the open immersion. Assume further that $j_\ast \mathcal{M}$ is trivialized by a Kummer covering $f: \Spec S \to \Spec R$ ramified along $x$. Then $Gr^0_V j_\ast \mathcal{M}$ is a unit $S/(t)[F]$-submodule of $Gr^0_V f^\ast j_\ast \mathcal{M}$.
\end{Le}
\begin{proof}
By \cite[Lemma 3.14]{stadnikvfiltrationfcrystal} the restriction of $V^0 f^\ast j_\ast \mathcal{M} \to F_\ast V^0 f^\ast j_\ast \mathcal{M}$ to $V^0 j_\ast \mathcal{M}$ (which is the intersection of $j_\ast \mathcal{M} \otimes R$ with $V^0 f^\ast j_\ast \mathcal{M}$) induces a map $V^0 j_\ast \mathcal{M} \to F_\ast V^0 j_\ast \mathcal{M}$. A similar statement holds for $V^\eps j_\ast \mathcal{M}$ and since $R/(x) \cong S/(t)$ we conclude that the unit $S/(t)[F]$-crystal structure on $Gr^0_V f^\ast j_\ast \mathcal{M}$ restricts to the unit $R/(x)[F]$-crystal structure on $Gr^0_V j_\ast \mathcal{M}$.
\end{proof} 

\begin{Prop}
\label{TestModuleVFiltrationKummerCase}
Let $R$ be smooth over an algebraically closed field. Let $x$ be a smooth element of $R$ and let $\mathcal{M}$ be a unit $F$-crystal on $D(x)$. Assume further that $\mathcal{M}$ is trivialized by a Kummer covering $f: \Spec S \to \Spec R$ ramified along $x$. Then \[V^a j_\ast \mathcal{M} \otimes \omega_R = \tau(j_\ast \mathcal{M} \otimes \omega_R, x^{a+1 - \eps}) \text{ for all } a > -1,\] where $0 < \eps \ll 1$ and $j: D(x) \to \Spec R$ is the open immersion. Moreover, the unit $R/(x)[F]$-crystal $Gr^0_V(j_\ast \mathcal{M})$ corresponds to the Cartier crystal $Gr^1_\tau(j_\ast \mathcal{M} \otimes \omega_R)$.
\end{Prop}
\begin{proof}
Assume that $S = \Spec R[t]/(t^n -x)$. In particular, $f$ has degree $n$ and $f^\ast \mathcal{M} \cong S_t^r$. By Proposition \ref{TestModuleVfiltrationShiftJustification} we have \[V^{a} j_\ast S_t^r \otimes \omega_S = \tau(j_\ast \omega_{S_t}^r, t^{a + 1 - \eps}).\]
Multiplying the right-hand side with $t^{n-1}$ and taking $G$-invariants we obtain $\tau(j_\ast \mathcal{M}, x^{\frac{a}{n} + 1 - \frac{\eps}{n}})$ by Propositions \ref{KummerTestModuleGinvariants} and \ref{ExponentAndRootTestModule}. By Lemma \ref{GInvariantsKummerTensor} the $G$-invariants $(t^{n-1}\tau(j_\ast \omega_{S_f}^r, t^{a+ 1 - \eps}))^G$ coincide with $(V^{a} j_\ast S_f^r)^G \otimes \omega_R$. Due to \cite[Lemma 3.14]{stadnikvfiltrationfcrystal} $(V^{a} j_\ast S^r_f)^G$ coincides with $V^{\frac{a}{n}} j_\ast \mathcal{M}$.\footnote{Our notation is slightly different from the one of Stadnik. In his notation $V^a j_\ast S_f^r$ should be $V^{\frac{a}{n}} j_\ast S_f^r$ (cf. \cite[Definition 3.1]{stadnikvfiltrationfcrystal}).} We conclude that $V^{\frac{a}{n}} j_\ast \mathcal{M} \otimes \omega_R = \tau(j_\ast \mathcal{M} \otimes \omega_R, x^{\frac{a}{n} + 1 - \eps})$ as claimed.

For the addendum recall that by the proof of Proposition \ref{TestModuleVfiltrationShiftJustification} one has $Gr_V^0 (f^\ast j_\ast \mathcal{M}) = (R/(x))^r$ and $Gr_\tau^1 f^! (j_\ast \mathcal{M} \otimes \omega_R) = (\omega_R/(x)\omega_R)^r \cong \omega_{R/(x)}^r$, where the latter isomorphism is due to Lemma \ref{ClosedImmersionCartierStructureGraded}.
We have the following situation
\[\begin{xy}
 \xymatrix{i^! i_\ast Gr_\tau^1(f^! (j_\ast \mathcal{M} \otimes \omega_R))  \ar[r]^{\cong} & i^\ast i_\ast Gr_V^0 (f^\ast j_\ast \mathcal{M}) \otimes \omega_{R/(x)}\\
 i^! i_\ast Gr^1_\tau(j_\ast \mathcal{M} \otimes \omega_R) \ar[u] \ar[r]^\varphi &  i^\ast i_\ast Gr^0_V(j_\ast \mathcal{M}) \otimes \omega_{R/(x)}  \ar[u]}
\end{xy}\]
where the vertical arrows are inclusions by Lemmata \ref{KummerGradedTestmoduleInclusion} and \ref{KummerGradedVFiltrationInclusion}. We only have to find an isomorphism $\varphi$ that makes this diagram commutative. 

We have the following chain of natural isomorphisms that induce both $\varphi$ and the top horizontal isomorphism.
By the first part of the proposition and the exactness of $\otimes \omega_R$ we have an isomorphism $i_\ast Gr^1_\tau(j_\ast \mathcal{M} \otimes \omega_R) \cong i_\ast Gr^0_V (j_\ast \mathcal{M}) \otimes \omega_R$. Applying $i^!$ to this isomorphism and using \cite[Corollary III.7.3]{hartshorneresidues} we get an isomorphism $i^! i_\ast Gr^1_\tau(j_\ast \mathcal{M} \otimes \omega_R) \cong i^\ast (i_\ast Gr_V^0(j_\ast \mathcal{M}) \otimes \omega_R) \otimes \omega_i$, where $\omega_i$ is by definition $\Hom_{R/(x)}((x)/(x^2), R/(x))$. Applying the isomorphism $i^\ast \omega_R \otimes \omega_i \cong \omega_{R/(x)}$ (cf.\ \cite[Proposition III.1.2]{hartshorneresidues}) we obtain $\varphi$ and the top horizontal isomorphism.
\end{proof}

\begin{Le}
\label{ShriekEtaleAndFiniteFlatCartesian}
Let $f: \Spec S \to \Spec R$ be a finite flat morphism and let $j: \Spec T \to \Spec R$ be \'etale. If $(M, \kappa)$ is a Cartier module on $R$ then we have a natural isomorphism of Cartier modules $j'^! f^! M \to f'^! j^! M$, where $j', f'$ denotes the base changes of $j, f$.
\end{Le}
\begin{proof}
We have to verify that the isomorphism \begin{align*}\Hom_R(S, M) \otimes_S (S \otimes_R T) &\longrightarrow \Hom_{T}(S \otimes_R T, M \otimes_R T)\\ \varphi \otimes (s \otimes t) &\longmapsto [s' \otimes t' \mapsto \varphi(ss') \otimes tt']\end{align*} is Cartier linear.

An element $\varphi \otimes (s^p \otimes t^p)$ is mapped to $\kappa \circ \varphi \circ F \otimes s \otimes t$ by the Cartier structure which is mapped to $[s' \otimes t' \mapsto \kappa(\varphi(ss'^p) \otimes tt'^p]$. The other way around we get that $[s' \otimes t' \mapsto \varphi(s's^p) \otimes t't^p]$ is mapped to $[s' \otimes t' \mapsto \kappa(\varphi(ss'^p) \otimes t't^p]$ as desired.
\end{proof}

\begin{Bem}
The result of Proposition \ref{TestModuleVFiltrationKummerCase} does not hold for $a \leq -1$. To see this consider $R$ and $f: \Spec S \to \Spec R$ as above with $n = p-1$ and take $M = j_\ast j^! \omega_R$ with Cartier structure given by $\kappa x$. Then $f^! j_\ast j^! \omega_R \cong j'_\ast j'^! f^! \omega_R = j'_\ast j'^\ast \omega_S$ by Proposition \ref{PushPullStuff} and Lemma \ref{ShriekEtaleAndFiniteFlatCartesian}, where the Cartier structure is given by $\kappa t^{p-1}$ (here $\kappa$ now denotes the natural Cartier structure on $\omega_S$).

First we claim that $\tau(\omega_S, t^0) = t \omega_S$ (where $\omega_S$ carries the Cartier structure $\kappa  t^{p-1}$). By Lemma \ref{PushforwardLocallyConstantSheafLNil} $\omega_S$ is $F$-pure. Since $t$ is a test element one has $\tau(\omega_S, t^0) = (\kappa t^{p-1})^e t \omega_S = \kappa^e t^{p^e} \omega_S = t \omega_S$ for $e \gg 0$, where we used \cite[Lemma 4.1]{blicklestaeblerbernsteinsatocartier}. A small computation shows that the inclusion $\omega_S \to j_\ast j^! \omega_{S}$ is a local nil-isomorphism for the Cartier structure given by $\kappa t^{p-1}$.

We conclude that for $a+1-\eps$ negative $\tau(j'_\ast j'^! \omega_S, t^{a + 1 -\eps}) = t^{\lfloor a +1 -\eps \rfloor} t \omega_S$, where we use right-continuity. In particular, working locally so that $\omega_R$ is free of rank $1$ we have $\tau(j'_\ast j'^! \omega_S, t^{a+1 -\eps}) = S \cdot t^{\lfloor a +1 -\eps \rfloor} \delta_{n -2}$, where $\delta_{n-2}(t^i) = \delta_{i, n-2} \omega$ for $0 \leq i \leq n-1$ and $\omega$ is a generator of $\omega_R$.

Now assume that $p =3$ (so $n =2$) and $-1 \geq a \in \mathbb{Z}$. We claim that \begin{equation}\label{wasd}\tau(j_\ast j^! \omega_R, x^{\frac{a - \eps}{n} +1}) \neq (t^{n-1}\tau(j'_\ast j'^! \omega_S, t^{a + 1 -\eps}))^G.\end{equation}
Take $a = -2$ (in fact, any even negative integer works) then the left-hand side is equal to $x^{-1} \omega_R$ while the right-hand side coincides with $(\omega_S)^G = R \delta_0 = \omega_R$. Inequality \eqref{wasd} shows that Proposition \ref{KummerTestModuleGinvariants} does not hold for negative exponents. Since it holds for Stadnik's super-specializing $V$-filtration the two filtrations cannot coincide  for negative values in general.
\end{Bem}

\begin{Le}
\label{TestmoduleVfiltrationEtaleDescent}
Let $R$ be smooth over an algebraically closed field. Let $x$ be a smooth element of $R$ and let $\mathcal{M}$ be a unit $F$-crystal on $D(x)$. Let $f: \Spec S \to \Spec R$ be a surjective \'etale morphism. Assume that \[ V^a f^\ast j_\ast \mathcal{M} \otimes \omega_S = \tau(f^\ast j_\ast \mathcal{M} \otimes \omega_S, x^{a+1-\eps})\] for an $a \geq -1$.
Then we have the equality \[ V^a j_\ast \mathcal{M} \otimes \omega_R = \tau(j_\ast \mathcal{M} \otimes \omega_R, x^{a+1-\eps}).\]
\end{Le}
\begin{proof}
By \cite[Corollary 6.17]{staeblertestmodulnvilftrierung} we have $\tau(f^\ast (j_\ast \mathcal{M} \otimes \omega_R), x^{a+1 -\eps}) = f^\ast \tau(j_\ast \mathcal{M} \otimes \omega_R, x^{a+1-\eps})$ and by \cite[Lemma 3.9]{stadnikvfiltrationfcrystal} we have $f^\ast V^a j_\ast \mathcal{M} =  V^a f^\ast j_\ast \mathcal{M}$. Moreover, as $f$ is \'etale we have a natural isomorphism $f^\ast j_\ast \mathcal{M} \otimes \omega_S \cong f^\ast (j_\ast \mathcal{M} \otimes \omega_R)$ which induces isomorphisms $\tau(f^\ast (j_\ast \mathcal{M} \otimes \omega_R), x^{a+1 -\eps}) \cong \tau(f^\ast j_\ast \mathcal{M} \otimes \omega_S, x^{a+1 -\eps})$. By faithfully flat descent the claim follows.
\end{proof}

\begin{Le}
\label{GrEtaleDescent}
Let $R$ be smooth over an algebraically closed field and $x \in R$ a smooth element. Let $f: \Spec S \to \Spec R$ be an \'etale morphism and let $\mathcal{M}$ be a unit $R[F]$-module admitting a $V$-filtration along $x$. If the unit $S/(x)[F]$-crystal $Gr_V^0 f^\ast \mathcal{M}$ corresponds to the Cartier crystal $Gr^1_\tau f^\ast \mathcal{M} \otimes \omega_S$ then the unit $S/(x)[F]$-crystal $f^\ast Gr_V^0 \mathcal{M}$ corresponds to the Cartier crystal $f^\ast Gr^1_\tau \mathcal{M} \otimes \omega_R$.

If $f$ is also surjective then the unit $R/(x)[F]$-crystal $Gr_V^0 \mathcal{M}$ corresponds to the Cartier crystal $Gr^1_\tau \mathcal{M} \otimes \omega_R$.
\end{Le}
\begin{proof}
According to \cite[Lemma 3.9]{stadnikvfiltrationfcrystal} one has a natural isomorphism of unit $S/(x)[F]$-crystals $Gr_V^0 (f^\ast M) \cong f^\ast Gr_V^0( M)$. By \cite[Corollary 6.21]{staeblertestmodulnvilftrierung}\footnote{The $F$-regularity assumption is not needed since we consider the associated graded of the test module filtration.} we have an isomorphism of Cartier crystals $f^\ast Gr_\tau^1 \mathcal{M} \otimes \omega_R \cong Gr_\tau^1 f^\ast (M \otimes \omega_R)$. Since the latter is naturally isomorphic to $Gr_\tau^1 f^\ast M \otimes \omega_S$ we obtain the desired isomorphism.

The last claim follows from faithfully flat descent.
\end{proof}

Now we can state the main result of this section (recall the definition of a tame unit $R[F]$-crystal -- Definiton \ref{TameUnitFCrystal}).

\begin{Theo}
\label{VFiltTestmoduleMainComparison}
Let $R$ be smooth over an algebraically closed field. Let $V(x) \subseteq \Spec R$ be a smooth hypersurface with complement $U$ and natural inclusion $j:U \to \Spec R$. If $j_\ast \mathcal{M}$ is a tame unit $R[F]$-crystal, then for all $t \geq -1$ we have \[V^t j_\ast \mathcal{M} \otimes \omega_R= \tau(j_\ast \mathcal{M} \otimes \omega_R, x^{t +1 -\eps}) \] for all $0 < \eps \ll 1$ depending on $t$, where $V^t$ denotes the $V$-filtration of Stadnik (cf.\ Definition \ref{StadnikVFiltration}) and $\tau$ is the extended test module filtration (cf.\ Definition \ref{LimitTestmodulefiltration}). Moreover, the unit $R/(x)[F]$-crystal $Gr^0_V(j_\ast \mathcal{M})$ corresponds to the Cartier crystal $Gr^1_\tau(j_\ast \mathcal{M} \otimes \omega_R)$.
\end{Theo}
\begin{proof}
The first claim is clearly Zariski local. By Lemma \ref{TestmoduleVfiltrationEtaleDescent} we may pass to \'etale coverings. Hence, we may assume that $j_\ast \mathcal{M}$ is trivialized by a Kummer covering by \cite[Lemma 3.12]{stadnikvfiltrationfcrystal}. But now the first claim follows from Proposition \ref{TestModuleVFiltrationKummerCase}. The second claim follows from combining Lemma \ref{GrEtaleDescent} and Proposition \ref{TestModuleVFiltrationKummerCase}. 
\end{proof}
\bibliography{bibliothek.bib}
\bibliographystyle{amsplain}
\end{document}